\documentclass{LMCS}

\def\dOi{10(2:13)2014}
\lmcsheading%
{\dOi}
{1--23}
{}
{}
{Sep.~\phantom04, 2012}
{Jun.~24, 2014}
{}


\ACMCCS{[{\bf Theory of computation}]: Logic---Modal and temporal logics}

\usepackage{amssymb}
\usepackage{amsmath}
\usepackage{latexsym}
\usepackage{comment}
\usepackage{leqno}
\usepackage{hyperref}

\newcommand{\rool}[3]{\frac{\textstyle #1}{\textstyle #2} {#3}}

\begin{document}

\title[Probability Logic]{Probability Logic for Harsanyi Type Spaces}

\author[C.~Zhou]{Chunlai Zhou}	
\address{Key Lab for Data and Knowledge Engineering, MOE,
   Department of Computer Science and Technology, School of Information,
   Renmin University of China, Beijing 100872, CHINA}	
\email{czhou@ruc.edu.cn}  
\thanks{This research is partly supported by
NSF of China (Grant Number: 60905036), by Key project for basic research from the Ministry of Science and Technology
of China (Grant Number: 2012CB316205) and by German Academic Exchange Service (DAAD Codenumber: A/08/94918).}	







\begin{abstract}
  \noindent Probability logic has contributed to significant developments in belief types for game-theoretical economics.  We present a new probability logic for Harsanyi Type spaces, show its completeness, and prove both a denesting property and a unique extension theorem.  We then prove that multi-agent interactive epistemology has greater complexity than its single-agent counterpart by showing that if the probability indices of the belief language are restricted to a finite set of rationals and there are finitely many propositional letters, then the canonical space for probabilistic beliefs with one agent is finite while the canonical one with at least two agents has the cardinality of the continuum. Finally, we generalize the three notions of  definability in multimodal logics to logics of probabilistic belief and knowledge, namely implicit definability, reducibility, and explicit definability.  We find that S5-knowledge can be implicitly defined by probabilistic belief but not reduced to it and hence is not explicitly definable by probabilistic belief.

\end{abstract}

\maketitle

\section{Introduction}

Probability logic plays an important role in computer science \cite{LS91, DEP02,Panangaden09, Doberkat10,Koz81, MorganMS96,HermannsPSWZ11},
 artificial intelligence \cite{Halpern91, FHM90, FH94, Shoham09} and
 economics \cite{AH02, HM01}.
 Game theoretical economics and multi-agent systems in artificial intelligence motivate our logical exploration of Harsanyi type spaces. In the context of  Harsanyi type spaces,   the introspection condition  in the multi-agent setting of game theory states that each agent
 is
certain (in the sense of belief with probability 1) of his degree of belief at every state.  In this paper, we  develop a logical theory of knowledge and belief for multi-agent systems in  Harsanyi type spaces by applying  techniques from probability logic \cite{Zhou09, Zhou10, Zhou11, ZhouY12}.

 Here we are concerned only about
 \emph{probabilistic beliefs} such as the
statement ``I believe that the chance of rain today is at least
ninety percent." This belief is about another statement that it will
rain today. Mathematically, statements are modeled as events. Since
the first statement involves probabilities, it is natural to
consider these statements in the context of a measurable space $S
=\langle \Omega, \mathcal{A}\rangle$ where $\Omega$ is a state space
and $\mathcal{A}$ is a $\sigma$-algebra on this space. So we
interpret this quantified belief using an operator $B^{\geq 0.9}$ on
$\mathcal{A}$, i.e. a mapping from $\mathcal{A}$ to $\mathcal{A}$.
  If $A$ stands for the event that
it will rain today, then $B^{\geq 0.9}A$ denotes the belief that the
chance of rain today is at least ninety percent. For any index $r
\in \mathbb{Q}\cap[0, 1]$ where $\mathbb{Q}$ is the set of
rationals, we define the operator $B^{\geq r}$  similarly. For such a
family of belief operators (also called probability operators) $B^{\geq r}$, one can easily check that such a family of belief operators  satisfies
 the following properties \cite{Sam00}: for events $A, A_n\in
\mathcal{A}$,

\begin{eqnarray}
    B^{\geq 0}(A) & = & \Omega\label{1}\\
    B^{\geq 1}(\Omega) & = & \Omega \\
    B^{\geq r} A & \subseteq & \sim B^{\geq s} (\sim A), r+s > 1\\
    r_n \uparrow r & \Rightarrow & B^{\geq r_n} A \downarrow B^{\geq r} A\label{4}\\
    B^{\geq r}(A\cap B) \cap B^{\geq s}(A\cap (\sim B)) & \subseteq & B^{\geq r+s}A \label{5}\\
    \sim B^{\geq r}(A\cap B) \cap \sim B^{\geq s}(A\cap (\sim B)) & \subseteq & \sim B^{\geq r+s} A\label{6}\\
    A_n  \downarrow A & \Rightarrow & B^{\geq r} A_n \downarrow B^{\geq r}
    A\label{7}
\end{eqnarray}

 \noindent where $\uparrow (\downarrow)$ means infinitely approaching
by an (a) increasing (decreasing) sequence and $\sim$ means set-complement.  The first three
properties say that degrees of beliefs are always between 0 and 1. Properties
(\ref{5}) and (\ref{6}) state that belief operators are finitely
additive. Property (\ref{7}) is the \emph{continuity from above property}
from measure theory. These three properties (\ref{5})-(\ref{7}) ensure that belief
operators are $\sigma$-additive.  Property $(\ref{4})$ states that these
operators are continuous in degrees (indices $r$ in $B^{\geq r}$). This
kind of treatment of belief interpreted by operators is analogous to
the treatment of knowledge in Kripke structures, interpreted by
partition-induced operators. As in the ``possible-world" semantics for knowledge, we need a
quantified version of ``Kripke structures" for probabilistic beliefs.
Such quantified Kripke structures, or \emph{belief
types},
 play a major role in game-theoretical economic theory.

 There are two approaches to belief types  of a multi-agent setup in game theory.
The first represents beliefs explicitly and is called
\emph{explicit description} of beliefs \cite{MZ85, Hei93}. Such a
description starts with a space of states of nature, which specifies
parameters of a game, e.g. payoff functions.  Next it specifies
the beliefs of the agents about the space of states of nature, and
then the beliefs about the combination of the nature space with the
beliefs about the nature space and so on.   An explicit belief type
consists of a hierarchy of beliefs which satisfy the
\emph{coherence} requirement that different levels of beliefs of
every agent do not contradict one another.  In the first layer,
beliefs are represented by probability measures over the nature
space, and beliefs in the second layer are represented by
probability measures over the space of probability measures in the
first layer, and so on. Such a straightforward description provides
all possible belief types, which form an explicit model for beliefs.
However, this model is hardly a workable model considering the
complexity of the representations of the belief types in it.

 The second approach describes  beliefs \emph{implicitly},
what we use for giving formal semantics for beliefs in this paper. This approach
was introduced by Harsanyi in 1960's \cite{Har67} for games with
incomplete information played by Bayesian players. The corresponding description
is defined in a measurable space $S=\langle \Omega,
\mathcal{A}\rangle$. For each agent $i$, we associate each state of this space
  with a state of nature and a probability measure on the
space.  The agent's \emph{implicit} belief type at the state is this
probability measure, which provides a belief over the nature space.
Since each state is associated with a belief type, the probability measure also defines
beliefs of beliefs about the nature space, and so on. So we can
extract a hierarchy of beliefs (or simply an explicit description) from
this implicit description. If we ignore the association with the
nature space, namely the economic content,  the above association
$T_i$ of states to probability measures is called a \emph{type
function}  from $\Omega$ to the measure space $\Delta(S)$ of
probability measures on $S$. We call the triplet $\langle \Omega,
\mathcal{A}, T_i\rangle$ \emph{a type space}.
 Samet\cite{Sam00} demonstrated  a natural one-to-one
correspondence between such defined type spaces and the families
above of belief operators for every agent. Type spaces provide exactly the expected
``quantified Kripke structures" for beliefs.

 Within the multi-agent setting of game theory,  a special kind of type spaces called \emph{Harsanyi type spaces}
are used to describe introspection of agents. In a Harsanyi type space,
 each agent is
\emph{certain} of his degree of belief at every state of this space.  In this paper,
we take  Aumann's approach to investigate properties about Harsanyi type spaces.

 Aumann gave a syntactic approach to beliefs \cite{Aum99b}, which is an alternative to the above
semantic approach. The building blocks of his syntactic formalism are
\emph{formulas}, which are constructed from propositional letters
(which are interpreted by Aumann as ``natural occurrences", in
contrast with the nature space in Harsanyi type spaces) by the
Boolean connectives  and a family of belief operators $L^i_r$ where
$r\in \mathbb{Q}\cap [0, 1]$, where $\mathbb{Q}$ is the set of
rational numbers.  The characteristic feature
of the syntax is this family of operators. We interpret
$L^i_r\phi$ as the statement the agent $i$'s belief in the event $\phi$ is at
least $r$. $L^i_r$ is the syntactic counterpart for the agent $i$ of the semantic
belief operator $B^{\geq r}$ on $\sigma$-algebras.   We also use a derived modality $M^i_r$ which means
``at most" in our semantics and is defined as
$M^i_r\phi:=L^i_{1-r}\neg\phi$. In this paper, when we focus on
reasoning about beliefs of one agent or on reasoning about probabilities without the multi-agent setting,
 we omit
the label $i$ for simplicity.

We give a complete axiomatization of probability logic in the above language
for Harsanyi type spaces. This system is based on our work in \cite{Zhou09} about the deductive
 system $\Sigma_+$ for the general type spaces.  The most important
principle of our axiomatization $\Sigma_+$ is an infinitary Archimedean rule:
\begin{center}
    $(ARCH): \rool{\gamma\rightarrow L_s\phi(s < r)}{\gamma\rightarrow L_r\phi}{\mbox .}$
\end{center}

\noindent It reveals that such a simple rule corresponds exactly to the property (\ref{4}) and characterizes reasoning about probabilities.  Except where indicated otherwise, completeness in this paper refers to weak completeness. That is, if a formula $\phi$ is valid, then it is provable.





The system $\Sigma_H$ for
Harsanyi type spaces is the basic probability logic $\Sigma_+$ plus the following two axiom schema:
\begin{itemize}
    \item $(4_p):  L_r \phi \rightarrow L_1L_r\phi$;
    \item $(5_p): \neg L_r \phi \rightarrow L_1\neg L_r\phi$
\end{itemize}
 which capture the introspection condition in the multi-agent setting that each agent is certain of his degree of beliefs at every state.

Using the above deductive machinery for Harsanyi type spaces, we  demonstrate the relative complexity of the multi-agent interactive epistemology compared with the one-agent epistemology from the perspective of probabilistic beliefs. In order to generalize different results about knowledge in interactive epistemology to the probabilistic setting, in the following sections we restrict  the probability indices in our language for type spaces to  a finite set of rationals and the propositional letters to a finite set.
 With respect to knowledge,  Aumann \cite{Aum99a} showed  that, when there are at least two agents and only one propositional letter, the canonical state space for knowledge has the cardinality of the continuum while, for a single agent, the cardinality of the canonical space is finite. Hart, Heifetz and Samet proved the same proposition in \cite{HartHS96} by exhibiting a continuum of different \emph{consistent} lists which are constructed through a new knowledge operator called ``knowing whether."  In this paper, we obtain a similar result for probabilistic beliefs by showing that, when there are at least 2 agents,  the canonical state space for probabilistic beliefs has the cardinality of the continuum while, when there is a single agent, the cardinality of the canonical space is finite. The first part is proved by adapting a counterexample in \cite{HartHS96} through a \emph{measure-theoretical} argument.  A straightforward proof idea  is that, if the knowledge operators used in \cite{HartHS96} are replaced by certainty operators (belief operators with probability 1), then the proof in \cite{HartHS96} goes through. Note that the axiomatization for reasoning about certainty is the logic $KD45$ \cite{Halpern91}, which is different from the logic $S_5$ for knowledge.  We show that this is true and obtain a much stronger result that the proof in \cite{HartHS96} goes through iff the degrees of belief operators are strictly bigger than $\frac{1}{2}$.  The second part is shown by demonstrating both a denesting property and a unique extension theorem that a certain kind of formulas are equivalent to formulas of depth $\leq 1$ and each maximally consistent set of formulas in the above restricted finite language  has only one maximal consistent extension in this finite language extended by increasing the depth by 1. Our denesting property and unique extension theorem demonstrate an important property about the probability logic $\Sigma_H$ for Harsanyi type spaces, which is  parallel to but weaker than the well-known denesting property in $S_5$ that any formula in $S_5$ is equivalent to a formula of depth 1.   Note that \emph{even} probability logic with probability indices in a finite set of rationals is essentially different from qualitative reasoning such as multi-modal logics. For example, this restricted probability logic does not have compactness \cite{Zhou10} while most modal logics do. So our techniques in the above two parts differ from those in \cite{HartHS96,Aum99a} for knowledge.

Knowledge is closely related to probabilistic belief in interactive epistemology \cite{Aum99b, FH94}.  In this paper we investigate the relationship between these two concepts in the context of the logic $\Sigma_{HK}:=(\Sigma_H + (S_5)_K + \{H_1, H_2, H_3\})$ where $(S_5)_K$ denotes the $S_5$-axiomatization for the  knowledge operator $K$, and
\begin{itemize}
    \item $(H_1): L_r\phi \rightarrow KL_r\phi$;
    \item $(H_2): \neg L_r\phi \rightarrow K\neg L_r\phi$;
    \item $(H_3): K\phi \rightarrow L_1\phi$.

\end{itemize}

\noindent Intuitively $H_1$ combined with $H_2$ says that  the agent knows the probabilities by which he believes events; $(H_3)$ says
that if he knows something, then he is certain of it (in the sense of belief with probability 1). It is well-known in interactive epistemology \cite{Aum99b, AH02} that,  although the concept of knowledge is \emph{implicit in probabilistic beliefs} in the above semantic framework, syntactically it is a separate and exogenous notion which is non-redundant.  Indeed, the probability syntax can only express beliefs of the agents, not that an agent knows anything about another one for sure.  In this paper, we give a \emph{logical} characterization of this relationship  between  probabilistic belief and knowledge by generalizing the three notions of definability in multi-modal logics \cite{HSS09a, HSS09b} to the setting of probabilistic beliefs and knowledge to show that this relationship is equivalent to the statement that semantically $\Sigma_{HK}$ determines knowledge \emph{uniquely} while syntactically there is no formula $DK$ in $\Sigma_{HK}$ of the form $Kp \leftrightarrow \delta$, where $\delta$ is a formula that does not contain any knowledge operator.  In other words, this relationship can also be expressed  in terms of definability in multimodal logic that
 knowledge is \emph{implicitly defined by probabilistic belief}, but is not \emph{reducible to it} and hence is not \emph{explicitly definable} in terms of probabilistic belief.

  The main purpose of this paper is to develop a logical theory of knowledge and belief for multi-agent systems in a \emph{probabilistic} setting by applying  techniques from probability logic.  Our main contributions  are as follows:

 \begin{enumerate}
    \item We present an axiomatization $\Sigma_H$ for the class of Harsanyi type spaces, which is different from those in the literature in our infinitary Archimedean rule (ARCH), and we employ a probabilistic  filtration method to show its completeness. By using the same method, we show in a restricted finite language with only finitely many probability indices both a denesting property and a unique extension theorem for $\Sigma_H$, which is a probabilistic version of the denesting property for $S_5$.
    \item We introduce a type of belief operators called ``strongly believing whether", which are a generalization of ``knowing whether" operator \cite{HartHS96}, and exhibit a continuum of different consistent lists of beliefs of two agents which are constructed through these ``strongly believing whether" operators; with the above unique extension theorem for $\Sigma_H$, we demonstrate the relative complexity of multi-agent interactive epistemology from the perspective of probabilistic beliefs.
    \item We provide a logical characterization of the relationship between knowledge and probabilistic beliefs by generalizing the three notions of definability in multimodal logic to logics of probabilistic beliefs and knowledge.
 \end{enumerate}

\noindent The rest of this paper is organized as follows. In the next section, we review the syntax and semantics  of probability logics and the completeness of the basic probability logic $\Sigma_+$. In Section 3, we show the completeness of the system $\Sigma_H$ with respect to the class of Harsanyi type spaces and obtain some basic properties about $\Sigma_H$. Moreover, we use the above semantic framework and deductive machinery to compute and compare the cardinalities of the canonical state spaces with only one agent and with many agents, respectively.   In Section 4, we generalize the three notions of definability: implicit definability, reducibility and explicit definability in multimodal logic to the setting of the logic $\Sigma_{HK}$ of probabilistic beliefs and knowledge, and show that knowledge is implicitly defined in $\Sigma_{HK}$ but is not reducible to probabilistic belief and hence is not explicitly definable in $\Sigma_{HK}$.
The paper concludes with some further open problems about the relationship between knowledge  and belief in Aumann's knowledge-belief systems.

\section{Basic Probability Logic $\Sigma_+$}

In this section, we briefly review the syntax and semantics for probability logics, and the completeness of the
basic probability logic $\Sigma_+$ \cite{HM01, FH94,Zhou09}.

\subsection{Syntax and Semantics}

 The syntax of our logic is very similar to that of modal
logic.   We start with a fixed countably infinite set $AP:=\{p_1, p_2,
\cdots\}$ of propositional letters. We also use $p, q, \cdots$ to
denote propositional letters.  The set of formulas $\phi$  is built
from propositional letters as usual by connectives $\neg$, $\wedge$
and a countably infinite modalities $L_r$ for each $r\in
\mathbb{Q}\cap [0, 1]$, where $\mathbb{Q}$ is the set of rational
numbers. Equivalently, a formula $\phi$ is formed by the following
syntax:

\begin{center} $\phi:= p\mid \neg \phi\mid\phi_1\wedge \phi_2\mid
L_r\phi$ ($r\in \mathbb{Q}\cap [0, 1]$)

\end{center}

 $L_r$ is the primitive modality in our language. But we
also use a derived modality $M_r$ which means ``at most" in our
semantics through the following definition:

\begin{center} (DEF M) {    } $M_r\phi := L_{1-r}\neg \phi$.

\end{center}

 Let $\mathcal{L}$ be the formal language consisting of the
above components. We use $r, s, \alpha, \beta, \cdots$ (also with
subscripts) to denote rationals.  Next we describe the semantics of
our system. A \emph{probability model} is a tuple

\begin{center}

$M:=\langle \Omega, \mathcal{A}, T, v\rangle$

\end{center}

 \noindent where

\begin{itemize}

\item $\Omega$ is a non-empty set, which is called \emph{the
universe or the carrier set} of $M$;

\item $\mathcal{A}$ is a $\sigma$-field (or $\sigma$-algebra) of subsets of $\Omega$;

\item $T$ is a measurable mapping from $\Omega$ to the space
$\Delta(\Omega, \mathcal{A})$ of probability measures on $\Omega$,
which is endowed with the $\sigma$-field generated by the sets:

\begin{center} $\{\mu\in \Delta(\Omega, \mathcal{A}): \mu(E)\geq
r\}$ for all $E\in \mathcal{A}$ and rational $r\in [0,
1]$,

\end{center}

\item $v$ is a mapping from $AP$ to $\mathcal{A}$, i.e. $v(p)\in
\mathcal{A}$.

\end{itemize}

  $\langle \Omega, \mathcal{A}, T\rangle$ is called a
\emph{type space}, and $T$ is called a \emph{type function} on the
space.

\begin{rem}  Let $T$ be a type function.  Define $k(w,
A):=T(w)(A)$ for any $w\in \Omega$ and $A\in \mathcal{A}$.  It is
easy to check that $k$ is a Markovian kernel, i.e., it satisfies the
following two conditions:

\begin{enumerate}

    \item $k(w, \cdot)$ is a probability measure on the
    $\sigma$-algebra $\mathcal{A}$ for any $w\in \Omega$;

    \item $k(\cdot, A)$ is an $\mathcal{A}$-measurable function for
    any $A\in \mathcal{A}$.

\end{enumerate}

 Conversely, if $k$ is a Markovian kernel on $(\Omega,
\mathcal{A})$, then every function $T: \Omega\rightarrow
\Delta(\Omega, \mathcal{A})$ such that $T(s)(A):=k(s, A)$ for all
$s\in \Omega$ and $A\in \mathcal{A}$ is a type function \cite{Zhou07}.





Since $T$ can be regarded as a Markovian kernel, we sometimes also write $T(w)(A)$ as $T(w, A)$.

\end{rem}

\begin{lem} \label{MainLemma}Assume that
    \begin{itemize}
        \item $\langle \Omega, \mathcal{A}\rangle$ is a measurable space, $\mathcal{A}_0$ is an algebra on $\Omega$ and $\mathcal{A}_0$ generates $\mathcal{A}$;
        \item $T: \Omega\times \mathcal{A}\rightarrow [0, 1]$ is a function satisfying the condition: for any $w\in \Omega$, $T(w, \cdot)$ is a probability measure.
    \end{itemize}
Then the following two statements are equivalent:
    \begin{enumerate}
        \item For any $r\in [0,1]$ and $E\in \mathcal{A}_0$, $B^{\geq r}(E):=\{w: T(w, E)\geq r\}\in \mathcal{A}$;
        \item For any $r\in [0,1]$ and $E\in \mathcal{A}$, $B^{\geq r}(E):=\{w: T(w, E)\geq r\}\in \mathcal{A}$

    \end{enumerate}
\end{lem}
The proof of this lemma is relegated to the Appendix.

 For a fixed model $M$, there exists a unique
satisfaction relation $\models$ between the state $w$ of $M$ and
formulas $\phi$ that satisfies the following clauses;
moreover, the associated interpretation $[[\phi]] = \{w\in M: M,
w\models \phi\}$ is a measurable set for all formulas $\phi$:

\begin{itemize}

\item $M, w\models p$ iff $w\in v(p)$ for propositional letters $p$;

\item $M, w\models \phi_1\wedge \phi_2$ iff $M, w\models \phi_1$ and
$M, w\models \phi_2$;

\item $M, w\models \neg \phi$ iff $M, w\not\models \phi$;

\item $M, w\models L_r\phi$ iff $T(w)([[\phi]])\geq r$, where $[[\phi]]:=\{w\in \Omega: M, w\models
\phi\}$.

\end{itemize}




\noindent A formula  $\phi$ is \emph{valid} in the probability model $M$ if $M\models
\phi$, i.e., for all states $w\in M$, $M, w\models \phi$. $\phi$ is
\emph{valid in a class of probability models} $\mathcal{C}$ if, for
each $M\in \mathcal{C}$, $M\models \phi$. $\phi$ is \emph{valid in a
class $\mathcal{T}$ of type spaces} if $\phi$ is valid in all the
probability models defined on $\mathcal{T}$.

\subsection{Completeness}

In this section, we give an axiomatization of the basic probability logic which is complete
 with respect to the class of type spaces and briefly review the proof of its completeness \cite{Zhou07, Zhou09}. Our following
proof of the completeness of probability logic for Harsanyi type spaces is based on this basic one. Without further notice, all the probability indices such as $r,s,t$ in this paper are assumed to be between 0 and 1.

\begin{center}

\textbf{Probability Logic $\Sigma_+$}

\begin{itemize}

\item (A0) propositional calculus

\item (A1) $L_0\phi$

\item (A2) $L_r \top$

\item (A3) $L_r(\phi\wedge \psi) \wedge L_t(\phi\wedge \neg \psi)
\rightarrow L_{r+t} \phi$,  for $r+t\leq 1$

\item (A4) $\neg L_r(\phi\wedge \psi) \wedge \neg L_s(\phi\wedge\neg
\psi)\rightarrow \neg L_{r+s}\phi,$ for $r+s\leq 1$

\item (A5) $L_r\phi\rightarrow \neg L_s \neg \phi,$ for $ r+s > 1$

\item (DIS) If $\vdash \phi\leftrightarrow \psi, \vdash L_r\phi\leftrightarrow
L_r\psi$.

\item (ARCH): If $\vdash \gamma\rightarrow L_s \phi$ for all $s < r$,
then $\vdash \gamma\rightarrow L_r\phi$.

\end{itemize}

\end{center}

 A \emph{logic} $\Lambda$ in $\mathcal{L}$ is a set of formulas in $\mathcal{L}$ that contains all propositional tautologies, and is closed under modus ponens and uniform
substitution (that is, if $\phi$ belongs to $\Lambda$, then so do
all of its substitution instances). For a formula $\phi$, $\vdash_{\Lambda} \phi$ denotes $\phi\in \Lambda$.  When the context is clear, we omit the subscript $\Lambda$.
 \emph{Probability logic $\Sigma_+$} is the smallest logic
that contains
($A1$-$A5$), and is closed under (DIS) and (ARCH). Observe that the rule (ARCH) is
the only rule that is really about the indices of the modalities.
Since the index set $\mathbb{Q}\cap [0,1]$ has  the
\emph{Archimedean property}, i.e., the property of having no
infinitely small elements, our following logics are outside nonstandard analysis and the fact that the rule has infinitely
many premises seems unavoidable.

\begin{defi} A finite set $\Theta =\{\phi_1, \cdots, \phi_k\}$
of formulas
 is \emph{inconsistent in a logic $\Lambda$} (or $\Lambda$-inconsistent) if
$\bigwedge_{i=1}^k \phi_i \rightarrow \bot \in \Lambda$, and is
\emph{consistent in $\Lambda$} (or $\Lambda$-consistent) otherwise.  A
(finite or infinite) set $\Xi$ of formulas is \emph{consistent in $\Lambda$} if each finite subset of $\Xi$ is consistent in $\Lambda$. $\Xi$ is
\emph{maximal consistent in $\Lambda$} if it is consistent in $\Lambda$ and
any set of formulas properly containing $\Xi$ is not.  $\Sigma_+$ is \emph{(weakly) complete} with respect to
the class of type spaces if any $\Sigma_+$-consistent formula is satisfiable in
a probability model.
\end{defi}

In the following sections, without further notice, we will omit the logic $\Lambda$ in $\Lambda$-consistency or $\Lambda$-inconsistency. The context will determine which logic we mean.
Note that consistency in the above definition is usually called finite consistency in probability logic.  Actually, in probability logic, finite consistency
is different from consistency \cite{Gol07, Zhou10}. But, since we won't discuss anything related to consistency in probability logic, there is no need to distinguish between these two notions. For simplicity, we call finite consistency in probability logic consistency as in the above definition.

\begin{lem}\label{Lindenbaum1} (Lindenbaum Lemma)

    If $\Xi$ is a consistent set of formulas, then there is a
maximal consistent $\Xi'$ such that $\Xi \subseteq \Xi'$.
\end{lem}

\begin{proof} Since our definition of consistency is the same as
that of consistency in modal logic \cite{BRV00}, so is the proof of
the Lindenbaum Lemma. The crucial step is the fact that  for any
$\Xi$ and $\phi$, if $\Xi$ is consistent, then one of
$\Xi\cup \{ \phi\}$ and $\Xi \cup\{\neg\phi\}$ is
consistent. For details, the reader may refer to Lemma 4.17
in \cite{BRV00}.
\end{proof}

Before we present the completeness, we will present some basic theorems
of $\Sigma_+$ which we will need in the following sections.  The reader can also read Section 3 in \cite{HM01}, Section 3 in \cite{Zhou09} and Chapter 2 in \cite{Zhou07} for detailed proofs.


\begin{prop}\label{basictheorem1} The following two principles are provable in $\Sigma_+$:

\begin{enumerate}

\item $\neg L_r \bot$ if $r > 0$;

\item $\neg M_r\top$ if $r<1$, which is just dual to the first part.

\end{enumerate}

\end{prop}

\begin{prop}\label{basictheorem2} The following principles are provable in  $\Sigma_+$:

 \begin{enumerate}

 \item If $ \phi\rightarrow \psi$, then $ L_r\phi\rightarrow
 L_r\psi$;

 \item $L_r\phi \rightarrow L_s\phi$ if $r\geq s$;


 \item $\neg L_r\phi \rightarrow M_r\phi$;

 \end{enumerate}

\end{prop}

\begin{prop} \label{basictheorem3}The following principles follow immediately from the above theorem:

 \begin{enumerate}

 \item If $\phi\rightarrow \psi$, then $M_r\psi\rightarrow
 M_r\phi$;

 \item $M_r\phi \rightarrow M_s\phi$ if $r\leq s$;


 \item $\neg M_r\phi \rightarrow L_r\phi$;

 \end{enumerate}

\end{prop}

\begin{prop}\label{additivity} The following principles hold:

\begin{enumerate}

\item  If $\vdash \neg (\phi\wedge \psi)$, then $\vdash L_r\phi\wedge
L_s\psi\rightarrow L_{r+s}(\phi\vee \psi),$ for $ r+s\leq 1$;

\item If $\vdash \neg(\phi\wedge \psi)$, then $\vdash \neg L_r\phi\wedge
\neg L_s\psi\rightarrow \neg L_{r+s}(\phi\vee \psi),$ for $ r+s\leq
1$.

\item  If $\vdash \neg (\phi\wedge \psi)$, then $\vdash M_r\phi\wedge
M_s\psi\rightarrow M_{r+s}(\phi\vee \psi),$ for $ r+s\leq 1$;

\item If $\vdash \neg(\phi\wedge \psi)$, then $\vdash \neg M_r\phi\wedge
\neg M_s\psi\rightarrow \neg M_{r+s}(\phi\vee \psi),$ for $ r+s\leq
1$.

\end{enumerate}

\end{prop}

\begin{cor} \label{PlusOne}
\begin{enumerate}
 \item $\vdash L_1\phi \wedge L_r\psi \rightarrow L_r(\phi \wedge \psi)$, $\vdash L_1\phi \wedge \neg L_r\psi \rightarrow \neg L_r(\phi \wedge \psi)$;
 \item $\vdash L_1\phi \wedge M_r\psi \rightarrow M_r(\phi \wedge \psi)$, $\vdash L_1\phi \wedge \neg M_r\psi \rightarrow \neg M_r(\phi \wedge \psi)$;
 \item $\vdash M_0\phi \wedge M_r\psi \rightarrow M_r(\phi \vee \psi)$, $\vdash M_0\phi \wedge \neg M_r\psi \rightarrow \neg M_r(\phi \vee \psi)$;
 \item $\vdash M_0\phi \wedge L_r\psi \rightarrow L_r(\phi \vee \psi)$, $\vdash M_0\phi \wedge \neg L_r\psi \rightarrow \neg L_r(\phi \vee \psi)$.
\end{enumerate}
\end{cor}
\proof  Here we choose Case 3 as an illustration:
\begin{eqnarray}
  M_0\phi \wedge M_r\psi & \rightarrow & M_0(\phi\wedge \neg \psi) \wedge M_r\psi \qquad \text{(Proposition \ref{basictheorem3})}\nonumber\\
 & \rightarrow & M_r(\phi \vee \psi)\qquad \text{(Proposition \ref{additivity})} \nonumber
\end{eqnarray}\vspace{-30 pt}

\qed

\begin{defi}  The \emph{depth} $dp(\phi)$ of a formula $\phi$ is defined
inductively:

\begin{itemize}

\item $dp(p):= 0$ for propositional letters $p$;

\item $dp(\neg \phi):= dp(\phi)$;

\item $dp(\phi_1\wedge \phi_2):= max\{dp(\phi_1), dp(\phi_2)\}$;

\item $dp(L_r \phi):=dp(\phi)+1$.

\end{itemize}

\end{defi}

 Now we define a local language $\mathcal{L}(q, d, P)$
 to be the largest set of formulas satisfying the following conditions:

\begin{itemize}

\item The propositional letters in $\mathcal{L}(q, d, P)$ are those occurring in $P$;

\item  The indices of formulas in $\mathcal{L}(q, d, P)$ are multiples
of $1/q$;

\item  The formulas in $\mathcal{L}(q, d, P)$ are of depth $\leq
d$;

\item Logical equivalent formulas are identified.

\end{itemize}

  The above $q$ is called the \emph{accuracy} of the
language $\mathcal{L}(q,d, P)$ ($q\in \mathbb{N}$).  In particular, $\mathcal{L}[\psi]$
is defined as $\mathcal{L}(q_{\psi}, d_{\psi}, P_{\psi})$ where
$P_{\psi}$ is the set of all propositional letters in $\psi$,
$q_{\psi}$ is the accuracy of $\psi$ (the least common multiple of all
denominators of the indices appearing in $\psi$) and $d_{\psi}$ the depth of
$\psi$. $I[\psi]$ is the finite set of all rationals in the form of
$p/q \in [0, 1]$  where $q$ is the accuracy of $\psi$; and it is called
the \emph{index set} of the language $\mathcal{L}[\psi]$. Note that
$\mathcal{L}[\psi]$ is finite. In general, let $\mathcal{L}(q,
d, P)$ be the set of formulas
 containing only propositional letters in $P$ having accuracy $q$ and depth at most $d$ modulo
 logical equivalence. From propositional reasoning, we know that
 each formula in
 $\mathcal{L}(q, d, P)$ is \emph{logically equivalent} to a finite disjunction
 consisting only of non-equivalent disjuncts, each of the disjuncts
 being a conjunction consisting only of non-equivalent conjuncts,
 each conjunct being either itself in $\mathcal{L}(q, d-1, P)$ or
 being obtainable from some formula in $\mathcal{L}(q, d-1, P)$ by
 prefixing it either with a modality $L_r$ or $\neg L_r$ where $r$
 is a multiple of $1/q$.  By induction on the depth $d$, we can show
 that $\mathcal{L}(q, d, P)$ is finite and hence $\mathcal{L}[\psi]$
 is finite. In the following, we won't distinguish between the equivalence
classes in $\mathcal{L}(q, d, P)$ and their representatives.

Assume that the set of propositional letters in the language $\mathcal{L}$ is implicitly enumerated as: $p_1, p_2, \cdots$.  If $P$ consists of the first $w$ propositional letters in the enumeration, then $\mathcal{L}(q, d, P)$ is also denoted as $\mathcal{L}(q, d, w)$ and $\Omega(q, d, w)$ denotes the set of all maximally consistent sets of formulas in $\mathcal{L}(q, d, w)$. Each element in $\Omega(q, d, w)$ is also called an \emph{atom}.

For any
maximal consistent set $\Xi$ in the finite local language
$\mathcal{L}'$ with accuracy $q'$ (and the index set $I'$) and depth $d'(\geq 1 )$, and for any
formula $\phi$ of depth $\leq d'-1$ in $\mathcal{L}'$, we define:

\begin{center} $\alpha_{\phi}^{\Xi}=max\{r\in I':
L_{r}\phi\in \Xi\}$ and $\beta_{\phi}^{\Xi}=min\{r\in
I': M_{r}\phi\in \Xi\}$.

\end{center}

\begin{lem}\label{density}For the above defined $\alpha_{\phi}^{\Xi}$ and $\beta_{\phi}^{\Xi}$, either
$\alpha_{\phi}^{\Xi}=\beta_{\phi}^{\Xi}$ or
$\beta_{\phi}^{\Xi}=\alpha_{\phi}^{\Xi}+1/q'$.
\end{lem}

Fix a formula $\psi$. Let us denote by $\Omega[\psi]$ the set of all maximally
consistent sets of formulas in $\mathcal{L}[\psi]$. This will be the carrier set of
our following canonical model. Since $\Omega[\psi]\subseteq 2^{\mathcal{L}[\psi]}$ and
$\mathcal{L}[\psi]$ is finite, so is $\Omega[\psi]$.  For any formula $\phi\in \mathcal{L[\psi]}$, define $[\phi]:=\{\Xi\in \Omega[\psi]: \phi\in \Xi\}$.

For any $\Gamma\in \Omega[\psi]$, there is a
maximal  consistent set $\Gamma^{\infty}$ in $
\mathcal{L}$ such that $\Gamma\subseteq
\Gamma^{\infty}$(Lemma \ref{Lindenbaum1}).  Note that such a
maximally consistent extension might not be unique.
Each such $\Gamma^{\infty}$ determines a
finitely additive probability measure on $\langle \Omega[\psi], 2^{\Omega[\psi]}\rangle$.  We define, for any $\phi\in
\mathcal{L}[\psi]$,

\begin{center}

$\alpha_{\phi}^{\Gamma^\infty}=sup\{r\in \mathbb{Q}:
L_{r}\phi\in \Gamma^{\infty}\} $ and
$\beta_{\phi}^{\Gamma^\infty}=inf\{r\in \mathbb{Q}:
M_{r}\phi\in \Gamma^{\infty}\} $

\end{center}

 Both $\alpha_{\phi}^{\Gamma^\infty}$ and $\beta_{\phi}^{\Gamma^\infty}$
might be irrational.

\begin{lem}
$\alpha_{\phi}^{\Gamma^\infty}=\beta_{\phi}^{\Gamma^\infty}$.

\end{lem}
\begin{lem}\label{total}

    \begin{enumerate}

        \item $2^{\Omega[\psi]} = \{[\phi]: \phi\in \mathcal{L}[\psi]\}$;

        \item For any $\phi_1, \phi_2\in \mathcal{L}[\psi]$,
    $\vdash_{\Sigma_+} \phi_1\rightarrow \phi_2$ iff $[\phi_1]\subseteq
    [\phi_2]$.

    \end{enumerate}

\end{lem}

 We define $T_{\psi}(\Gamma): 2^{\Omega[\psi]}\rightarrow [0,
1]$ by
\begin{center}
$T_{\psi}(\Gamma)([\phi]) = \alpha_{\phi}^{\Gamma^\infty}$.
\end{center}

Such a defined function $T_{\psi}: 2^{\Omega[\psi]}\rightarrow
[0, 1]$ is well-defined. From the first part of Lemma \ref{total},  it follows that $T_{\psi}(\Gamma)$ is
total. It is easy to see that
$T_{\psi}(\Gamma)(\Omega[\psi])=T_{\psi}(\Gamma)([\top])=1$ since $L_1\top\in
\Gamma^{\infty}$.

\begin{thm}\label{finite-additivity} For $A, B\in 2^{\Omega[\psi]}$, if
$A\cap B=\emptyset$, then $T_{\psi}(\Gamma)(A)+T_{\psi}(\Gamma)(B)=T_{\psi}(\Gamma)(A\cup
B)$; hence such a defined $T_{\psi}$ defines a probability measure on $\langle \Omega[\psi], 2^{\Omega[\psi]}\rangle$.

\end{thm}

The following proposition is the crucial step to show the truth lemma. Also this is the place where we really need the rule $(ARCH)$.

\begin{thm} \label{Crucial} For any  $\Gamma\in \Omega[\psi]$,
there is a maximally consistent extension $\Gamma^{\infty}$ in $\mathcal{L}$ of $\Gamma$ such that,
for any formula $\phi\in \mathcal{L}[\psi]$,
\begin{itemize}
    \item if $\alpha_{\phi}^{\Gamma} =\beta_{\phi}^{\Gamma}$, then $ \alpha_{\phi}^{\Gamma^{\infty}}= \alpha_{\phi}^{\Gamma}$;
    \item if $\alpha_{\phi}^{\Gamma} < \beta_{\phi}^{\Gamma}$, then $ \alpha_{\phi}^{\Gamma} < \alpha_{\phi}^{\Gamma^{\infty}}< \beta_{\phi}^{\Gamma}$;

\end{itemize}

\end{thm}

For each $\Gamma \in \Omega[\psi]$ and its such extension $\Gamma^{\infty}$, we define a type function on $\langle \Omega[\psi], 2^{\Omega[\psi]}\rangle$: $T_{\psi}(\Gamma)([\phi]):=\alpha_{\phi}^{\Gamma^{\infty}}$. It is easy to check that $M[\psi]:=\langle \Omega[\psi], 2^{\Omega[\psi]}, T_{\psi}\rangle$ is a probability model. We call it \emph{a canonical model} for $\Sigma_+$.

\begin{lem}\label{truth} (Truth Lemma) For any $\phi\in \mathcal{L}[\psi]$ and $\Gamma \in \Omega[\psi]$,

\begin{center}

$M[\psi], \Gamma\models \phi$ iff  $\phi\in
\Gamma$.

\end{center}

\end{lem}

\begin{thm} \label{Completeness}(Completeness) For any formula $\psi$, it is consistent (in $\Sigma_+$) if and only if it is satisfied in the class of type spaces.

\end{thm}

\section{Probability Logic $\Sigma_H$ for Harsanyi Type Spaces}

\noindent In this section we will show that the axiom system
$\Sigma_+$ plus the following two axiom schemes:

\begin{itemize}

    \item $(4_p): L_r\phi\rightarrow L_1L_r\phi$

    \item $(5_p):  \neg L_r\phi\rightarrow L_1\neg
L_r\phi$

\end{itemize}

\noindent is sound and complete with respect to a special class of type spaces called \emph{Harsanyi
type spaces}.  Let $\Sigma_H$ denote this system and $\vdash_{\Sigma_H}\phi$ denote $\phi \in \Sigma_H$.  Assume that
$\langle \Omega, \mathcal{A}, T\rangle$ is a type space.
$[T(w)]$ denotes $\{w'\in \Omega: T(w)=T(w')\}$. Note that $[T(w)]$ is not necessarily
 $\mathcal{A}$-measurable. If each type is
certain of its type, i.e.,

\begin{center} $(H): T(w)(E)=1$ for all $w\in \Omega$ and $E$ such that $[T(w)] \subseteq E \in \mathcal{A}$,

\end{center}

\noindent the type space is called \emph{a Harsanyi type space}.   If $\mathcal{A}$
is generated by a countable subalgebra $\mathcal{A}_0$, then

\begin{eqnarray}
[T(w)] & = & \bigcap_{A\in\mathcal{A}_0}\{w'\in \Omega:
T(w)(A)=T(w')(A)\}\nonumber\\
 & = & \bigcap_{A\in \mathcal{A}_0}\bigcap_{r\in
\mathbb{Q}\cap[0, 1]}\{w'\in \Omega: T(w')(A)\geq r \leftrightarrow
T(w)(A) \geq r\}\nonumber\\
& = & \bigcap_{A\in \mathcal{A}_0}\bigcap_{r\in \mathbb{Q}\cap[0,
1], T(w)(A) \geq r}\{w'\in \Omega: T(w')(A)\geq r \}\nonumber
\end{eqnarray}

\noindent and hence $[T(w)]$ is $\mathcal{A}$-measurable.  With
this countable subalgebra assumption, it is easy to check that the condition $(H)$ is equivalent to the following condition:
\begin{center}
    $(H')$: For any $w\in \Omega$ and $E\in \mathcal{A}$, $T(w)(\{v\in \Omega: T(v, E)\neq T(w, E)\})=0$,
\end{center}
which is the defining condition in \cite{AH02} for Harsanyi type spaces.


In order to show the following properties of $\Sigma_H$,  we need to define \emph{normal forms} for probability formulas.

\begin{defi} Let $\Phi$ be a maximally consistent set of formulas in $\mathcal{L}(q, d,w)$, namely an atom.    A formula $\phi$  is called a \emph{statement} in $\mathcal{L}(q, d, w)$ if it is of the form $\bigwedge\{p_i: p_i\in \Phi\}\wedge
    \bigwedge \{\neg p_i: \neg p_i\in \Phi\}$ $\wedge \bigwedge \{L_r\phi: L_r\phi \in
    \Phi\} \wedge \bigwedge \{\neg L_r\phi:
    \neg L_r\phi\in \Phi\}$.  For the formula $\phi$, its first part  $\bigwedge\{p_i: p_i\in \Phi\}\wedge
    \bigwedge \{\neg p_i: \neg p_i\in \Phi\}$ is called \emph{the
    propositional part}  and its second part $\bigwedge \{L_r\phi: L_r\phi \in
    \Phi\} \wedge \bigwedge \{\neg L_r\phi:
    \neg L_r\phi\in \Phi\}$ is called \emph{the
    probability part}.

\end{defi}
 For a statement $\phi$ in $\mathcal{L}(q, d, w)$, its propositional part and  probability part are denoted by $\phi^{=0}$
 and $\phi^{>0}$, respectively.

\begin{lem} For any atom $\Gamma\in \Omega(q, d, w)$,  the conjunction $\gamma$ of all
formulas in $\Gamma$
 is tautologically equivalent to a statement in the language
 $\mathcal{L}(q, d, w)$, which is simply called \emph{the statement of the atom $\Gamma$}. The propositional (probability) part of the statement is simply called \emph{the propositional (probability) part of the atom $\Gamma$}.   Moreover, any formula $\phi\in \Phi(q, d,
 w)$ is tautologically equivalent to a disjunction of statements in
  $\mathcal{L}(q, d, w)$. This disjunction is called \emph{a normal form} of
  $\phi$.

\end{lem}

\begin{proof} This proposition follows from the same argument for disjunctive normal forms in the usual propositional reasoning.
\end{proof}

\begin{lem}\label{WedgeOne}  $\vdash_{\Sigma_H}(L_r\phi\wedge L_s \psi)\rightarrow L_1(L_r\phi\wedge L_s
\psi)$, and  $\vdash_{\Sigma_H}(L_r\phi\wedge \neg L_s
\psi)\rightarrow L_1(L_r\phi\wedge \neg L_s \psi)$

\end{lem}

\proof  The proof of the second part is similar to that of
the first one.  We only show the first part.  Reason inside
$\Sigma_H$:

\begin{eqnarray}
L_r\phi \wedge L_s\psi & \rightarrow & L_1L_r\phi \wedge
L_1L_s\psi \qquad \text{(Axiom $4_p$)}\nonumber\\
& \rightarrow & L_1(L_r\phi \wedge L_s\psi)\qquad \text{(Corollary \ref{PlusOne})}\nonumber
\end{eqnarray}\vspace{-30 pt}

\qed

\begin{cor} \label{CertainProbability}For any statement $\phi$, $\vdash_{\Sigma_H}\phi^{> 0}\rightarrow L_1(\phi^{>0})$.

\end{cor}

\subsection{Completeness}
The following completeness theorem is based on the proof of the completeness of $\Sigma_+$ (Theorem \ref{Completeness}):

\begin{thm} \label{HarsanyiCompleteness} $\Sigma_H$ is sound and complete with respect to the
class of Harsanyi type spaces.
\end{thm}

\noindent The soundness of the system is clear.   We concentrate on
the completeness part.  In Heifetz and Mongin \cite{HM01}, they gave a proof
sketch of the completeness proof. For the completeness of our presentation,  here we give a detailed proof whose style is different from that in \cite{HM01}.
  Assume that $\psi$ is
consistent. We need to show that it is satisfiable in a Harsanyi
type space. Just as in the proof of the completeness of $\Sigma_+$ (Theorem \ref{Completeness}),
we define a local language $\mathcal{L}[\psi]$.   It gives rise to
a set $\Omega_H[\psi]$ which consists of all maximal $\Sigma_H$-consistent sets, whose elements are also called
\emph{atoms} in $\mathcal{L}[\psi]$ whenever no confusion arises.


\begin{proof} Enumerate all the atoms in $\Omega_H[\psi]$:

\begin{center}
    $\Gamma_1, \Gamma_2, \cdots, \Gamma_{N[\psi]}$.

\end{center}

\noindent The conjunction of all formulas in each atom $\Gamma_i\in \Omega_H[\psi]$ is denoted as $\gamma_i$.  It has a
normal form or a statement $\gamma_i^{=0}\wedge \phi_i^{>0}$ where $\gamma_i^{=0}$ is
the propositional part and $\gamma_i^{>0}$ is the probability part.




   We divide the set of atoms in $\Omega_H[\psi]$ into
several groups $(G_k)_{k=1}^K$ according to their probability parts:

%



\begin{itemize}

    \item all the atoms in each group have the same probability
    parts;

    \item atoms in different groups have different probability
    parts.

\end{itemize}
Next we define probability measures at all atoms according
to their representatives at each group.  Pick up a representative $\Gamma_{i_k}$
from each group $G_k$.  We know from Theorem \ref{Crucial} that $\Gamma_{i_k}$ has a maximal consistent extension $\Gamma_{i_k}^{\infty}$ in
$\mathcal{L}$ such that
    \begin{itemize}
        \item $\Gamma_{i_k}\subseteq \Gamma_{i_k}^{\infty}$;
        \item   For any formula $\phi\in \mathcal{L}[\psi]$,
                    \begin{enumerate}
                        \item if $\alpha_{\phi}^{\Gamma_{i_k}} =\beta_{\phi}^{\Gamma_{i_k}}$, then $\alpha_{\phi}^{\Gamma_{i_k}^{\infty}} =
                            \alpha_{\phi}^{\Gamma_{i_k}}$;
                        \item if $\alpha_{\phi}^{\Gamma_{i_k}} < \beta_{\phi}^{\Gamma_{i_k}}$, then $\alpha_{\phi}^{\Gamma_{i_k}}<
                            \alpha_{\phi}^{\Gamma_{i_k}^{\infty}} < \beta_{\phi}^{\Gamma_{i_k}}$.

                    \end{enumerate}

    \end{itemize}\medskip

\noindent Let $T_{\psi}(\Gamma_{i_k})$ be the probability measure on $\langle \Omega_H[\psi], 2^{\Omega_H[\psi]}\rangle$ according to $\Gamma_{i_k}^{\infty}$ defined by

                    \begin{center} for any formula $\phi$ in $\mathcal{L}[\psi]$,   $T_{\psi}(\Gamma_{i_k})([\phi]) = \alpha_{\phi}^{\Gamma_{i_k}^{\infty}}$.
                    \end{center}

 For each $\Gamma\in G_k$, define $T_{\psi}(\Gamma): = T_{\psi}(\Gamma_{i_k})$. That is to say, the probability measures at atoms in each group $G_k$ are defined to be the same as the probability measure at the representative $\Gamma_{i_k}$.  The canonical model in $\mathcal{L}[\psi]$ for $\Sigma_H$ is $M_H[\psi]:=\langle \Omega_H[\psi], 2^{\Omega_H[\psi]}, T_{\psi}, V_{\psi}\rangle$ where $V_{\psi}(p) = [p]$ for any propositional letter $p$.

\begin{clm}  $T_{\psi}(\Gamma_{i_k})(G_k) =1$

\end{clm}

\noindent  Observe that $G_k = [\gamma_{i_k}^{>0}]$, the set of states in $M_H[\psi]$ containing $\gamma_{i_k}^{>0}$.  Since
$\Gamma^{\infty}_{i_k}$ is a maximal consistent extension
of $\Gamma_{i_k}$ and $\gamma^{i_k}_{>0}\rightarrow L_1(\gamma^{i_k}_{>0})\in \Sigma_H$ (Corollary \ref{CertainProbability}), $L_1(\gamma_{i_k}^{>0})\in
\Gamma^{\infty}_{i_k}$.  So $T_{\psi}(\Gamma_{i_k})(G_k)
=T_{\psi}(\Gamma_{i_k})([\gamma_{i_k}^{>0}])=
\alpha_{\gamma_{i_k}^{>0}}^{\Gamma^{\infty}_{i_k}} =1$.\\

This implies that the above defined canonical model $M_H[\psi]$ is indeed based on a Harsanyi type space.

\begin{clm} For any formula $\phi$ in $\mathcal{L}[\psi]$ and any atom
$\Xi\in \Omega_H[\psi]$,  $M_H[\psi], \Xi \models \phi$ iff $\phi\in \Xi$. Equivalently,
$[\phi]=[[\phi]]$.

\end{clm}

The proof of this claim is similar to that of Lemma \ref{truth} (or Lemma 3.17 in \cite{Zhou09}).
It follows immediately that $\psi$ is satisfiable in a Harsanyi type space.
\end{proof}

Just as $\Omega(q,d, w)$ in last section about $\Sigma_+$, $\Omega_H(q, d, w)$ denotes the collection of
maximally $\Sigma_H$-consistent sets of formulas in $\mathcal{L}(q, d, w)$. $M_H(q, d, w)$ denotes the canonical Harsanyi type on $\Omega_H(q, d, w)$. In the remainder of
this section, we will compute the cardinality of this space.

\subsection{Denesting Property and Unique Extension Theorem}

First we show a denesting property for $\Sigma_H$, which says that each \emph{normal} formula (which will be defined below) is equivalent to a formula of depth $\leq 1$.
 Next we prove a unique extension theorem in $\Sigma_H$ that each maximally consistent set of formulas in a finite local language has only one maximal consistent extension in this finite language extended by increasing the depth by 1. Our denesting property and unique extension theorem demonstrate an important property about the probability logic $\Sigma_H$ for Harsanyi type spaces which is  parallel to but weaker than the well-known denesting property in $S_5$ that any formula in $S_5$ is equivalent to a formula of depth 1 (Appendix in \cite{Aum99a}).

\begin{lem} The proposition consists of four parts:

\begin{enumerate}

\item  $\vdash_{\Sigma_H} L_s\psi \wedge L_r\phi\rightarrow L_r(\phi
\wedge L_s\psi)$

\item  $\vdash_{\Sigma_H} L_s\psi \wedge \neg L_r\phi\rightarrow \neg L_r(\phi
\wedge L_s\psi)$

\item  $\vdash_{\Sigma_H} \neg L_s\psi \wedge L_r\phi\rightarrow L_r(\phi
\wedge \neg L_s\psi)$

\item  $\vdash_{\Sigma_H} \neg L_s\psi \wedge \neg L_r\phi\rightarrow \neg L_r(\phi
\wedge \neg L_s\psi)$

\end{enumerate}
\end{lem}

\begin{proof}   Reason
inside $\Sigma_{H}$:

\begin{eqnarray}
L_r\phi\wedge L_s\psi & \rightarrow & L_r\phi \wedge  L_1(L_s\psi)
\nonumber\\
& \rightarrow &L_r(\phi\wedge L_s \psi) \qquad \text{(Corollary \ref{PlusOne})}\nonumber\\
\nonumber
\end{eqnarray}
Similarly, other parts follow from Corollary \ref{PlusOne}.
\end{proof}

\begin{thm}  The proposition consists of two parts:

\begin{enumerate}

\item  $\vdash_{\Sigma_H} (\bigwedge_{i=1}^m L_{s_i}\psi_i \wedge
\bigwedge_{j=1}^n \neg L_{t_j}\psi_j')\wedge L_r\phi\rightarrow
L_r(\phi \wedge (\bigwedge_{i=1}^m L_{s_i}\psi_i \wedge
\bigwedge_{j=1}^n \neg L_{t_j}\psi_j'))$

\item $\vdash_{\Sigma_H} (\bigwedge_{i=1}^m L_{s_i}\psi_i \wedge
\bigwedge_{j=1}^n \neg L_{t_j}\psi_j')\wedge \neg L_r\phi\rightarrow
\neg L_r(\phi \wedge (\bigwedge_{i=1}^m L_{s_i}\psi_i \wedge
\bigwedge_{j=1}^n \neg L_{t_j}\psi_j'))$

\end{enumerate}

\end{thm}

\proof This proposition follows directly from the above lemma
according to the following fact:

\begin{center} $\vdash_{\Sigma_H}(\bigwedge_{i=1}^m L_{s_i}\psi_i \wedge
\bigwedge_{j=1}^n \neg L_{t_j}\psi_j')\rightarrow L_1
(\bigwedge_{i=1}^m L_{s_i}\psi_i \wedge \bigwedge_{j=1}^n \neg
L_{t_j}\psi_j')$
\end{center}\vspace{-18 pt}

\qed

\begin{lem}\label{VeeOne} $\vdash_{\Sigma_H}(L_r\phi\vee L_s \psi)\rightarrow L_1(L_r\phi\vee L_s
\psi)$

\end{lem}

\proof  Reason inside $\Sigma_{H}$:

  \begin{eqnarray}
  L_r \phi & \rightarrow & L_1 L_r\phi\nonumber\\
           & \rightarrow & L_1 (L_r\phi \vee L_s\psi)\nonumber\\
  L_s \psi & \rightarrow & L_1 (L_r\phi \vee L_s\psi)\nonumber\\
  L_r\phi \vee L_s\psi & \rightarrow & L_1 (L_r \phi \vee L_s\psi)\nonumber
  \nonumber \end{eqnarray}\vspace{-30 pt}

\qed

\begin{thm} \label{PosIntro} If $\phi$ is a Boolean combination of formulas of the
form $L_r\psi$, then $\phi\rightarrow L_1\phi$ is a theorem in
$\Sigma_H$.
\end{thm}

\proof  Assume that $\phi$ is a Boolean combination of
formulas of the form $L_r\gamma$ and its disjunctive normal form is
$\bigvee_{i=1}^I \bigwedge_{j=1}^{k_i}  \pi^i_jL_{r^i_j} \phi^i_j$
where $\pi^i_j$ is either blank or $\neg$.   Reason inside
$\Sigma_H$:

\begin{eqnarray}
\bigvee_{i=1}^I \bigwedge_{j=1}^{k_i}  \pi^i_jL_{r^i_j} \phi^i_j &
\rightarrow & \bigvee_{i=1}^I L_1(\bigwedge_{j=1}^{k_i}
\pi^i_jL_{r^i_j} \phi^i_j)\qquad \text{(Corollary \ref{WedgeOne})}\nonumber\\
& \rightarrow & L_1(\bigvee_{i=1}^I \bigwedge_{j=1}^{k_i}
\pi^i_jL_{r^i_j} \phi^i_j)\qquad \text{(Lemma \ref{VeeOne})}\nonumber
\end{eqnarray}\vspace{-33 pt}

\qed

\begin{thm} \label{denesting1} If $\phi$ is a Boolean combination of formulas of the
form $L_r\psi$, then, for any $r>0$,  $\phi\leftrightarrow L_r\phi$ is a theorem in
$\Sigma_H$.
\end{thm}

\proof From the above theorem, it is easy to see that $\phi \rightarrow L_r \phi$ is a theorem in $\Sigma_H$.  For the other direction, we reason inside $\Sigma_H$:

\begin{eqnarray}
\neg \phi &
\rightarrow & L_1 \neg \phi  \qquad \text{(Theorem \ref{PosIntro})}\nonumber\\
& \rightarrow &  M_0  \phi  \nonumber\\
& \rightarrow & \neg L_r  \phi   \nonumber
\end{eqnarray}\vspace{-30 pt}

\qed

\begin{thm} \label{denesting2} If $\psi$ is a Boolean combination of formulas of the
form $L_r\psi'$, then, for any $r>0$ and any formula $\phi$,  $(L_r \phi \wedge \psi) \leftrightarrow L_r(\phi \wedge \psi)$ is a theorem in
$\Sigma_H$.
\end{thm}

\proof First we show
    \begin{eqnarray}
        L_r \phi \wedge \psi & \rightarrow & L_r \phi \wedge L_1 \psi \qquad \text{ (Theorem \ref{PosIntro})} \nonumber \\
        & \rightarrow & L_r(\phi \wedge \psi)  \qquad \text{(Corollary \ref{PlusOne})}\nonumber
    \end{eqnarray}
    For the other direction, we reason inside $\Sigma_H$:
    \begin{eqnarray}
    L_r(\phi \wedge \psi) & \rightarrow & L_r \phi \nonumber \\
    & \rightarrow & L_r \psi \nonumber \\
    & \rightarrow & \psi \qquad \text{(Theorem \ref{denesting1})}\nonumber\\
     L_r(\phi \wedge \psi) & \rightarrow & L_r \phi \wedge \psi \nonumber
 \end{eqnarray}\vspace{-30 pt}

\qed

\begin{thm}\label{denesting3} If $\psi$ is a Boolean combination of formulas of the
form $L_r\psi'$, then, for any $r>0$ and and formula $\phi$,  $(L_r \phi \vee \psi) \leftrightarrow L_r(\phi \vee \psi)$ is a theorem in
$\Sigma_H$.
\end{thm}

\proof First we show that $(L_r \phi \vee \psi) \rightarrow L_r(\phi \vee \psi)$ is a theorem in $\Sigma_H$. We reason inside $\Sigma_H$ as follows:
    \begin{eqnarray}
        L_r \phi & \rightarrow & L_r(\phi \vee \psi) \nonumber\\
        \psi & \rightarrow & L_1 \psi \qquad \text{(Theorem \ref{PosIntro})} \nonumber \\
        & \rightarrow & L_r \psi \nonumber \\
        & \rightarrow & L_r (\phi \vee \psi) \nonumber \\
        L_r \phi \vee \psi & \rightarrow & L_r (\phi \vee \psi) \nonumber
    \end{eqnarray}

    For the other direction, we reason as follows:
    \begin{eqnarray}
    \neg \psi & \rightarrow & L_1 \neg \psi \qquad \text{(Theorem \ref{PosIntro})} \nonumber \\
    & \rightarrow & M_0 \psi \nonumber\\
    \neg L_r \phi \wedge \neg \psi & \rightarrow & \neg L_r (\phi \vee \psi) \qquad \text{(Corollary \ref{PlusOne})} \nonumber
    \end{eqnarray}\vspace{-30 pt}

\qed

It follows from the above three denesting theorems that a certain kind of probability formulas is equivalent to a formula of depth $\leq 1$.  A formula $\phi$ in the language $\mathcal{L}$ is called \emph{normal} if it is obtained inductively as follows
\begin{enumerate}
    \item any formula of depth $\leq 1$ is normal;
    \item if $L_{r_1}\phi_1, \cdots, L_{r_n}\phi_n$ are normal, then any Boolean combination of $L_{r_1}\phi_1, \cdots, L_{r_n}\phi_n$ is normal;
    \item if $L_r\phi (r>0)$ is normal and $\psi$ is a Boolean combination of normal formulas of the form $L_s \psi' (s >0)$, then  $L_r \psi$, $L_r(\phi \wedge \psi)$ and $L_r(\phi \vee \psi)$ are normal.
\end{enumerate}

\begin{cor} Any normal formula is equivalent to a formula of depth $\leq 1$ in $\Sigma_H$.
\end{cor}

However, we don't have the following propositions in $\Sigma_H$:
\begin{enumerate}
    \item $L_r(\phi_1 \wedge \phi_2) \leftrightarrow L_{r_1}\phi_1 \wedge L_{r_2}\phi_2$
    \item $L_r(\phi_1 \vee \phi_2) \leftrightarrow L_{r_1}\phi_1 \vee L_{r_2}\phi_2$
\end{enumerate}

 \noindent which correspond to the normality of the knowledge operator $K: K(\phi_1 \wedge \phi_2) = K(\phi_1) \wedge K(\phi_2)$. Otherwise, with the above Theorems \ref{denesting1}, \ref{denesting2} and \ref{denesting3}, we would have been able to show that that any formula
 is equivalent to a formula of depth $\leq 1$.  Despite this defect, we will show Unique Extension Theorem which is similar to but weaker than the above two properties.

 In order to show the following Unique Extension Theorem, we prove a simple case as an illustration.  Recall that, for any formula $\phi$ in $\mathcal{L}(q, d, w)$ and any atom $\Gamma\in \Omega_H(q, d+1, w)$,
 \begin{center}
    \item $\alpha_{\phi}^{\Gamma} =max \{r: L_r\phi \in \Gamma\}$ and $\beta_{\phi}^{\Gamma} =inf \{r: M_r\phi \in \Gamma\}$.
\end{center}

\begin{lem} \label{SimpleExtension} Let $\Gamma_1, \Gamma_2$ and $\Gamma_3$ be three different atoms
in $\Omega_H(q, d, w)$ and $\gamma_1, \gamma_2$ and $\gamma_3$ be the equivalent statements of the
conjunctions of formulas in these three atoms, respectively.  Their
statements
 are $\gamma_1^{=0}\wedge \gamma_1^{>0}$, $\gamma_2^{=0}\wedge
\gamma_2^{>0}$ and $\gamma_3^{=0}\wedge \gamma_3^{>0}$.   Then the
following three propositions hold:

\begin{enumerate}

    \item if both  $\gamma_2^{>0}$ and $\gamma_3^{>0}$ are different
    from $\gamma_1^{>0}$ up to logical equivalence, then $\gamma_1\rightarrow M_0(\gamma_2\vee \gamma_3)$ is provable
    in $\Sigma_H$;

    \item if only one of them, say, $\gamma_2^{>0}$, is different from
    $\gamma_1^{>0}$ up to logical equivalence, then

        \begin{enumerate}
            \item $\gamma_1\rightarrow L_{\alpha_{\gamma_3^{=0}}^{\Gamma_1}}(\gamma_2\vee \gamma_3)
                \wedge M_{\beta_{\gamma_3^{=0}}^{\Gamma_1}}(\gamma_2\vee \gamma_3)$ is provable
                in $\Sigma_H$ whenever $\alpha_{\gamma_3^{=0}}^{\Gamma_1}=\beta_{\gamma_3^{=0}}^{\Gamma_1}$;

                \item $\gamma_1\rightarrow L_{\alpha_{\gamma_3^{=0}}^{\Gamma_1}}(\gamma_2\vee
                \gamma_3)\wedge \neg M_{\alpha_{\gamma_3^{=0}}^{\Gamma_1}}(\gamma_2\vee \gamma_3)
                \wedge M_{\beta_{\gamma_3^{=0}}^{\Gamma_1}}(\gamma_2\vee \gamma_3)
                \wedge \neg L_{\beta_{\gamma_3^{=0}}^{\Gamma_1}}(\gamma_2\vee \gamma_3)$ is provable
                in $\Sigma_H$ whenever $\alpha_{\gamma_3^{=0}}^{\Gamma_1} < \beta_{\gamma_3^{=0}}^{\Gamma_1}$;
        \end{enumerate}

    \item if none of these two is different from $\gamma_1^{>0}$,i.e., all of them are the same up to logical equivalence, then

    \begin{enumerate}
        \item $\gamma_1 \rightarrow L_{\alpha_{\gamma_2^{=0}\vee
            \gamma_3^{=0}}^{\Gamma_1}}(\gamma_2\vee \gamma_3) \wedge M_{\beta_{\gamma_2^{=0}\vee
            \gamma_3^{=0}}^{\Gamma_1}}(\gamma_2\vee \gamma_3)$ is a theorem of
            $\Sigma_H$ whenever $\alpha_{\gamma_2^{=0}\vee
            \gamma_3^{=0}}^{\Gamma_1} =\beta_{\gamma_2^{=0}\vee
            \gamma_3^{=0}}^{\Gamma_1}$;

        \item $\gamma_1 \rightarrow L_{\alpha_{\gamma_2^{=0}\vee
            \gamma_3^{=0}}^{\Gamma_1}}(\gamma_2\vee \gamma_3) \wedge \neg M_{\alpha_{\gamma_2^{=0}\vee
            \gamma_3^{=0}}^{\Gamma_1}}(\gamma_2\vee \gamma_3)\wedge M_{\beta_{\gamma_2^{=0}\vee
            \gamma_3^{=0}}^{\Gamma_1}}(\gamma_2\vee \gamma_3)\wedge \neg L_{\beta_{\gamma_2^{=0}\vee
            \gamma_3^{=0}}^{\Gamma_1}}(\gamma_2\vee \gamma_3)$ is a theorem of
            $\Sigma_H$ whenever $\alpha_{\gamma_2^{=0}\vee
            \gamma_3^{=0}}^{\Gamma_1} < \beta_{\gamma_2^{=0}\vee
            \gamma_3^{=0}}^{\Gamma_1}$;

    \end{enumerate}

\end{enumerate}

\end{lem}

\noindent In other words, the probability of the set $\{\Gamma_2, \Gamma_3\}$ is syntactically uniquely determined  at the state $\Gamma_1$ in the canonical model $M_H(q, d,w)$.
Intuitively, Part 1 says that, if neither $\Gamma_2$ nor $\Gamma_3$  belongs to the same group as $\Gamma_1$ (in terms of the proof of Theorem \ref{HarsanyiCompleteness}), then
the probability measure at $\Gamma_1$ should assign 0 to the event consisting of the two states $\Gamma_2$ and $\Gamma_3$.  Parts 2 and 3 can be explained in the same way.

\proof  Assume that both  $\gamma_2^{>0}$ and $\gamma_3^{>0}$ are different
    from $\gamma_1^{>0}$.   Reason inside $\Sigma_H$:

    \begin{eqnarray}
    \gamma_1 & \rightarrow & \neg \gamma_2^{>0} \qquad \text{($\Gamma_1$ and $\Gamma_2$ are disjoint from each other)}\nonumber\\
    & \rightarrow & L_1(\neg \gamma_2^{>0})\nonumber\\
    & \rightarrow & M_0(\gamma_2^{>0})\nonumber\\
    \gamma_1 & \rightarrow & M_0(\gamma_3^{>0})\nonumber\\
     \gamma_1 & \rightarrow & M_0(\gamma_2^{>0}\vee\gamma_3^{>0})\nonumber\\
     \gamma_1 & \rightarrow & M_0(\gamma_2\vee \gamma_3) \qquad \text{(Proposition \ref{basictheorem3})}\nonumber
    \end{eqnarray}

\noindent Next we assume that only one of them, say, $\gamma_2^{>0}$,
is different from $\gamma_1^{>0}$.   From above, we know that
$\gamma_1\rightarrow M_0 \gamma_2$ is provable in $\Sigma_H$.

  \begin{eqnarray}
    \gamma_1 & \rightarrow & \gamma_3^{>0}\nonumber\\
    & \rightarrow & L_1(\gamma_3^{>0}) \qquad \text{(Corollary \ref{PlusOne})}\nonumber\\
    & \rightarrow & L_{\alpha_{\gamma_3^{=0}}^{\Gamma_1}}\gamma_3^{=0} \qquad \text{(by definition of $\alpha_{\gamma_3^{=0}}^{\Gamma_1}$)} \nonumber\\
    & \rightarrow & L_{\alpha_{\gamma_3^{=0}}^{\Gamma_1}}(\gamma_3^{=0}
    \wedge \gamma_3^{>0})\nonumber\\
    & \rightarrow & L_{\alpha_{\gamma_3^{=0}}^{\Gamma_1}}(\gamma_3)\nonumber\\
    \gamma_1 & \rightarrow & M_0 \gamma_2\nonumber\\
    \gamma_1 & \rightarrow & L_{\alpha_{\gamma_3^{=0}}^{\Gamma_1}}(\gamma_2\vee\gamma_3) \qquad \text{(Corollary \ref{PlusOne})}\nonumber\\
     \gamma_1 & \rightarrow &
     M_{\beta_{\gamma_3^{=0}}^{\Gamma_1}}(\gamma_3)\nonumber\\
    \gamma_1 & \rightarrow & M_0 \gamma_2\nonumber\\
     \gamma_1 & \rightarrow &
     M_{\beta_{\gamma_3^{=0}}^{\Gamma_1}}(\gamma_2\vee \gamma_3)\qquad \text{(Corollary \ref{PlusOne})}\nonumber\\
     \gamma_1 & \rightarrow & L_{\alpha_{\gamma_3^{=0}}^{\Gamma_1}}(\gamma_2\vee \gamma_3)\wedge
     M_{\beta_{\gamma_3^{=0}}^{\Gamma_1}}(\gamma_2\vee \gamma_3)\nonumber\\
    \nonumber \end{eqnarray}

    \noindent   Assume that none of these two is different from
    $\phi^1_{>0}$.  Then $\gamma_1^{>0} =\gamma_2^{>0} =\gamma_3^{>0}$.
    Reason inside $\Sigma_H$:

    \begin{eqnarray}
    \gamma_1 & \rightarrow & \gamma_2^{>0}\nonumber \\
    & \rightarrow & L_1 (\gamma_2^{>0}) \nonumber\\
    \gamma_1 & \rightarrow & L_1 (\gamma_2^{>0})\nonumber \\
    \gamma_1 & \rightarrow & L_1 (\gamma_3^{>0}) \nonumber\\
     & \rightarrow & L_{\alpha_{\gamma_2^{=0}\vee
     \gamma_3^{=0}}^{\Gamma_1}}(\gamma_2^{>0}\wedge (\gamma_2^{=0} \vee
     \gamma_3^{=0}))\nonumber\\
     \gamma_1 & \rightarrow & L_{\alpha_{\gamma_2^{=0}\vee
     \gamma_3^{=0}}^{\Gamma_1}}((\gamma_2^{>0}\wedge \gamma_2^{=0}) \vee
     (\gamma_3^{>0}\wedge
     \gamma_3^{=0}))\nonumber\\
     \gamma_1 & \rightarrow & L_{\alpha_{\gamma_2^{=0}\vee
     \gamma_3^{=0}}^{\Gamma_1}}(\gamma_2\vee \gamma_3)\nonumber\\
     \gamma_1 & \rightarrow & M_{\beta_{\gamma_2^{=0}\vee
     \gamma_3^{=0}}^{\Gamma_1}}(\gamma_2\vee \gamma_3)\nonumber\\
     \gamma_1 & \rightarrow & L_{\alpha_{\gamma_2^{=0}\vee
     \gamma_3^{=0}}^{\Gamma_1}}(\gamma_2\vee \gamma_3)\wedge M_{\beta_{\gamma_2^{=0}\vee
     \gamma_3^{=0}}^{\Gamma_1}}(\gamma_2\vee \gamma_3)\nonumber
    \end{eqnarray}\vspace{-30 pt}

\qed
\medskip

\begin{thm}\label{UniqueExtension} (Unique Extension Theorem) Probabilities of formulas $\phi$ in the maximally consistent extensions $\Gamma(q, d,w)$  are \emph{uniquely} determined by probabilities of their propositional parts $\phi^{=0}$ in the restrictions $\Gamma(q, d, w) \cap \mathcal{L}(q, 1, w)$
within \emph{depth 1}.    Let $\Gamma(q, d+1, w)$ be a maximally consistent
extension of $\Gamma(q, d, w)\in \Omega(q, d, w)$ by increasing its
depth by 1 and the statement of the conjunction of all formulas in $\Gamma(q, d, w)$ be
$\gamma^{=0}\wedge\gamma^{>0}$. Define $\Gamma(q, 1, w):=\Gamma(q, d,
w)\cap \mathcal{L}(q, 1, w)$.  Assume that $\phi$ is a formula in $\mathcal{L}(q,
d, w)$ and its normal form is the disjunction of the
following statements:

\begin{center}

    $\phi_1^{=0} \wedge \phi_1^{>0}, \phi_2^{=0} \wedge \phi_2^{>0},
    \cdots, \phi_n^{=0} \wedge \phi_n^{>0}$.

\end{center}

\noindent In addition, assume that the first $m(\leq n)$
probability parts are the same as $\gamma^{>0}$ and other probability
parts are different (up to logical equivalence), i.e.

\begin{center} $\phi_i^{>0} =\gamma^{>0} (1\leq i\leq m)$ and $\phi_j^{>0} \neq\gamma^{>0} (m+1\leq j\leq n)$

\end{center}

\noindent Then  $\alpha_{\phi}^{\Gamma(q, d+1, w)} =
\alpha_{\bigvee_{i=1}^m \phi_i^{=0}}^{\Gamma(q, 1, w)}$ and
$\beta_{\phi}^{\Gamma(q, d+1, w)} = \beta_{\bigvee_{i=1}^m
\phi_i^{=0}}^{\Gamma(q, 1, w)}$.    That is to say,  $\Gamma(q, d,
w)$ has one and only one maximal consistent extension in $\Omega_H(q,
d+1, w)$, which is $\Gamma(q, d+1, w)$.

\end{thm}
\begin{proof} This theorem is a straightforward generalization of the above Lemma \ref{SimpleExtension}.
\end{proof}

It would be interesting to compare this unique extension theorem to the denesting property in $S5$.  The above theorem tells us that, in each maximally consistent set $\Gamma$ of the canonical models of any finite local language, the probability of any formula $\phi$ is uniquely determined by the probabilities of formulas of depth 0 in the subset of $\Gamma$ consisting of all formulas of depth $\leq 1$.  So,  it is similar to the denesting property in $S5$ that any formula in $S5$ is equivalent to a formula of depth 1.

The following theorem tells us that the unique extension is actually determined by the probability part.

\begin{thm}  Assume that $\Gamma_1(q, d, w)$ and $\Gamma_2(q, d,
w)(d\geq 1)$ are two atoms in $\Omega_H(q, d, w)$, and  $\Gamma_1(q,
d+k, w)$ and $\Gamma_2(q, d+k, w)$ are maximal consistent extensions
in $\Omega_H(q, d+k, w)$ of $\Gamma_1(q, d, w)$ and $\Gamma_2(q, d,
w)$, respectively, where $k$ is a natural number.  If the statements of $\Gamma_1(q, d,w)$ and
$\Gamma_2(q, d, w)$ have the same probability parts, then the statements of $\Gamma_1(q, d+k, w)$ and $\Gamma_2(q, d+k, w)$ also have
the same probability parts.

\end{thm}

\begin{cor}  Any atom $\Gamma(q, d, w)\in \Omega_H(q, d, w)
(d\geq 1)$ has one and only one maximal consistent extension in
$\Omega_H(q, d+k, w)$ for any $k\geq 0$.  More precisely,  it is the
probability part of the statement of $\Gamma(q, d, w)$ that
uniquely determines the probability part of its maximal
consistent extension in $\Omega_H(q, d+k, w)$.

\end{cor}

\subsection{Cardinality of the Canonical Model $M_H(q, d, w)$}

In order to compute the cardinality of the canonical model $M_H(q, d, w)$, we first show a ``conservation" result in the sense that a formula of depth $\leq 1$ provable in $\Sigma_H$ is also provable in $\Sigma_+$.

\begin{thm} \label{ConservationDepth} For any formula $\psi$ of depth $\leq 1$,
$\vdash_{\Sigma_H} \psi$ if and only if $\vdash_{\Sigma_+} \psi$.

\end{thm}

\begin{proof} It suffices to show that, for any formula of depth 1, if
it is consistent in $\Sigma_+$, then so is it in $\Sigma_H$.  Assume
that $\psi$ is a $\Sigma_+$-consistent formula of depth 1.  Then it
is contained in a maximal $\Sigma_+$-consistent set $\Gamma_0(q(\psi), 1,
w(\psi))\in \Omega(q(\psi), 1, w(\psi))$.
Recall that $\Omega(q(\psi), 1, w(\psi))$ denotes the set of all maximal $\Sigma_+$-consistent
set of formulas in $\mathcal{L}(q(\psi), 1, w(\psi))$ and $\Gamma_0(q(\psi), 0, w(\psi))$ denotes the set
$\Gamma_0(q(\psi), 1, w(\psi))\cap \mathcal{L}(q(\psi), 0, w(\psi))$, which is the set of formulas of
depth 0 in $\Gamma_0(q(\psi), 1, w(\psi))$.\\

\noindent  Now we define a Harsanyi type space on $\Omega(q(\psi), 0, w(\psi))$.
Consider $\Gamma_0(q(\psi), 0, w(\psi))$, which is an element of $\Omega(q(\psi), 0,
w(\psi))$.  We know from Theorem \ref{Crucial} that there is a probability measure $T_{\psi}(\Gamma(q(\psi),
0, w(\psi)))$ at $\Gamma(q(\psi), 0, w(\psi))$, which is defined through a maximal \emph{$\Sigma_+$-consistent} extension $\Gamma^{\infty}(q(\psi), 0, w(\psi))$ in $\mathcal{L}$,  such that, for any formula $\phi\in
\mathcal{L}(q(\psi), 0, w(\psi))$,

\begin{enumerate}

    \item if $\alpha_{\phi}^{\Gamma_0(q(\psi), 1, w(\psi))} =  \beta_{\phi}^{\Gamma_0(q(\psi), 1,
    w(\psi))}$, then $T_{\psi}(\Gamma_0(q(\psi), 0, w(\psi)))([\phi]) = \alpha_{\phi}^{\Gamma_0(q(\psi), 1, w(\psi))}
    $ where, as usual, $[\phi] = \{\Xi\in \Omega(q(\psi), 0, w(\psi)): \phi \in \Xi\}$;

    \item if $\alpha_{\phi}^{\Gamma_0(q(\psi), 1, w(\psi))} <  \beta_{\phi}^{\Gamma_0(q(\psi), 1,
    w(\psi))}$, then $\alpha_{\phi}^{\Gamma_0(q(\psi), 1, w(\psi))}< T_{\psi}(\Gamma_0(q(\psi), 0, w(\psi)))([\phi]) < \beta_{\phi}^{\Gamma_0(q(\psi), 1, w(\psi))}
    $.

\end{enumerate}

\noindent For other atoms $\Xi\in \Omega(q(\psi), 0, w(\psi))$,
define $T_{\psi}(\Xi): =  T_{\psi} $ $(\Gamma_0(q(\psi), 0, w(\psi)))$ and further the
canonical model $M(q(\psi), 0, w(\psi)):=\langle \Omega(q(\psi), 0, w(\psi)), 2^{\Omega(q(\psi),
0, w(\psi))}, T_{\psi}, V\rangle$ where $V(p)=[p](: = \{\Xi\in \Omega(q(\psi),
0, w(\psi)): p\in \Xi\})$.  $M(q(\psi), 0, w(\psi))$ is a Harsanyi type
space. Indeed, for any atom $\Xi\in \Omega(q(\psi), 0, w(\psi))$, $[T_{\psi}(\Xi)]
= \Omega(q(\psi), 0, w(\psi))$ and hence $T_{\psi}(\Xi)([T_{\psi}(\Xi))])=1$.

\begin{clm} \label{Claim3} $M(q(\psi), 0, w(\psi)), \Gamma_0(q(\psi), 0, w(\psi))\models \psi$.

\end{clm}

\noindent It suffices to show that, for any formula $\phi$ in $\mathcal{L}(q(\psi),
1, d(\psi))$,

\begin{center}
$M(q(\psi), 0, d(\psi)), \Gamma_0(q(\psi), 0, d(\psi))\models \phi$ iff $\phi\in
\Gamma_0(q(\psi), 1, d(\psi))$.
\end{center}
 It is easy to check that this is true for the
base case and the Boolean cases.  The proof of the nontrivial case: $\phi = L_r\phi'$
is similar to that of Lemma \ref{truth}. For completeness of the presentation, we provide the proof details here.

Here we only prove the crucial case: $\phi = L_r\phi'$.  Assume that
$\Gamma_0(q(\psi), 0, d(\psi)) \models L_r\phi'$. According to  our above definition of $\Gamma_{\psi}$ and induction hypothesis: $[[\phi']]=[\phi']$, $T_{\psi}(\Gamma_0(q(\psi), 0, d(\psi)))([\phi']) \geq r$.  If $\alpha_{\phi'}^{\Gamma_0(q(\psi), 1, d(\psi))}=\beta_{\phi'}^{\Gamma_0(q(\psi), 1, d(\psi))}$, then $T_{\psi}(\Gamma_0(q(\psi), 0, d(\psi)))$ $([\phi']) = \alpha_{\phi'}^{\Gamma_0(q(\psi), 1, d(\psi))}\geq r$ and  $L_r\phi' \in \Gamma_0(q(\psi), 1, d(\psi))$.  If $\alpha_{\phi'}^{\Gamma_0(q(\psi), 1, d(\psi))}<\beta_{\phi'}^{\Gamma_0(q(\psi), 1, d(\psi))}$, then $\alpha_{\phi'}^{\Gamma_0(q(\psi), 1, d(\psi))}<T_{\psi}(\Gamma_0(q(\psi), 0, d(\psi)))([\phi'])< \beta_{\phi'}^{\Gamma_0(q(\psi), 1, d(\psi))}$. Since $r$ is a multiple of $1/q_{\psi}$ and $\beta_{\phi'}^{\Gamma_0(q(\psi), 1, d(\psi))}=\alpha_{\phi'}^{\Gamma_0(q(\psi), 1, d(\psi))}+1/q_{\psi}$ (Lemma \ref{density}),  $r\leq \alpha_{\phi'}^{\Gamma_0(q(\psi), 1, d(\psi))}$ and hence $L_r\in \Gamma_0(q(\psi), 1, d(\psi))$.

For the other direction, assume that $L_r\phi'\in \Gamma_0(q(\psi), 1, d(\psi))$. This implies that $L_r\phi'\in \Gamma^{\infty}(q(\psi), 0, w(\psi))$, which is a maximal consistent extension of
 $\Gamma_0(q(\psi), 1, d(\psi))$ in $\mathcal{L}$, and $T_{\psi}(\Gamma_0(q(\psi), 0, d(\psi)))([\phi]) = \alpha_{\phi'}^{\Gamma^{\infty}} \geq r$ by the definition of $\alpha_{\phi'}^{\Gamma^{\infty}}$.  According to the induction hypothesis, $M[\psi], \Gamma_0(q(\psi), 0, d(\psi))\models \phi$.


Since $\psi\in \Gamma_0(q(\psi), 1, w(\psi))$, $\psi$ is satisfied at $\Gamma_0(q(\psi), 0, w(\psi))$  in $M(q(\psi), 0, w(\psi))$.  According to the soundness for $\Sigma_H$, $\psi$ is $\Sigma_H$-consistent.
\end{proof}

\setlength{\unitlength}{2mm}
    \begin{picture}(80,45)

    \put(30, 41){$\top$}
    \put(30, 40){\vector(-4, -1){20}}
    \put(30, 40){\vector(-2, -1){10}}
    \put(30, 40){\vector(4, -1){20}}
    \put(30, 40){\vector(2, -1){10}}

    \put(5, 33){$\Gamma_1(q, 1, w)$}
    \put(15, 33){$\Gamma_2(q, 1, w)$}
    \put(35, 33){$\Gamma_{N-1}(q, 1, w)$}
    \put(47, 33){$\Gamma_N(q, 1, w)$}

    \put(10, 32){\vector(0, -1){5}}
    \put(20, 32){\vector(0, -1){5}}
    \put(40, 32){\vector(0, -1){5}}
     \put(50, 32){\vector(0, -1){5}}

     \put(5, 25){$\Gamma_1(q, 2, w)$}
    \put(15, 25){$\Gamma_2(q, 2, w)$}
    \put(35, 25){$\Gamma_{N-1}(q, 2, w)$}
    \put(47, 25){$\Gamma_N(q, 2, w)$}

     \put(10, 22){$\vdots$}
    \put(20, 22){$\vdots$}
    \put(40, 22){$\vdots$}
    \put(50, 22){$\vdots$}

      \put(10, 20){\vector(0, -1){5}}
    \put(20, 20){\vector(0, -1){5}}
    \put(40, 20){\vector(0, -1){5}}
     \put(50, 20){\vector(0, -1){5}}

      \put(5, 12){$\Gamma_1(q, d, w)$}
    \put(15, 12){$\Gamma_2(q, d, w)$}
    \put(35, 12){$\Gamma_{N-1}(q, d, w)$}
    \put(47, 12){$\Gamma_N(q, d, w)$}

     \put(30, 35){$\cdots$}
     \put(30, 23){$\cdots$}
     \put(30, 13){$\cdots$}
     
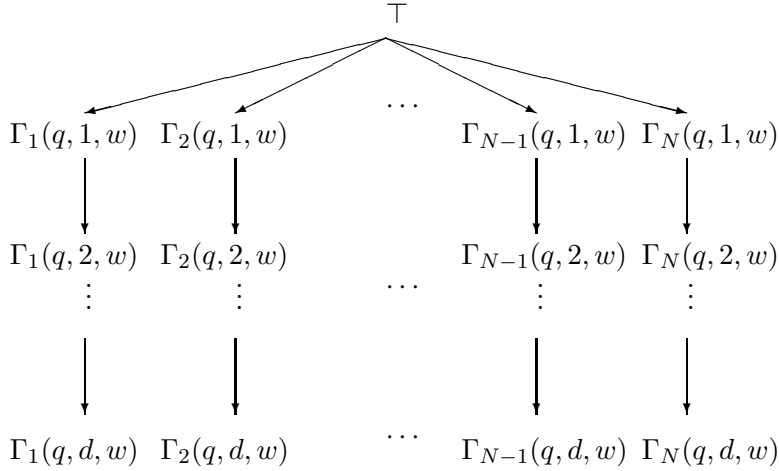
     \put(15, 5){Figure 1: maximal consistent
    extensions }

    \end{picture}

    \noindent The above figure illustrates the maximal consistent
    extensions in $\Sigma_H$.   The first step maximal consistent
    extensions from $\top$ in $\Sigma_H$ are the same as those in
    $\Sigma_+$ because the set of maximal $\Sigma_H$-consistent
    sets of formulas of depth $\leq 1$ is the same as that of maximal
    $\Sigma_+$-consistent sets of formulas of depth $\leq 1$.   But
    after that, any atom has only one maximal consistent extension,
    which is illustrated in the figure by demonstrating that each node from the
    second step has only one descendant.

\begin{thm}\label{Independence} If $\gamma^{=0}$ is the propositional part of some atom $\Gamma$
in $\Omega_H(q, d, w)$ and $\delta^{>0}$ is the probability part of some atom $\Xi$ in $\Omega_H(q, d, w)$, then
$\gamma^{=0}\wedge \delta^{>0}$ is $\Sigma_H$-consistent.  In this sense, the propositional part and the probability part are independent of each other.

\end{thm}

\begin{proof} Under the above assumption, let $\Xi_1$ be the set of formulas in $\Xi$ of the form $L_r p$ or $\neg L_r p$ where $p$ is a propositional letter.  First we show that $\{\gamma^{=0}\}\cup\Xi_1$ is $\Sigma_H$-consistent. But this follows directly from the same argument as in Claim \ref{Claim3} of the above Theorem \ref{ConservationDepth}.

It is easy to see that $\{\gamma^{=0}\}\cup\Xi_1$ can be extended to a maximally $\Sigma_H$-consistent atom $\Gamma_1\in \Omega_H(q, 1, w)$ and the statement of $\Gamma_1$ is tautologically equivalent to the conjunction of all formulas in $\{\gamma^{=0}\}\cup\Xi_1$.  According to the above Unique Extension Theorem, $\Gamma_1$
 has a unique maximal consistent extension in $\Omega_H(q, d, w)$ whose statement must be tautologically equivalent to $\gamma^{=0}\wedge \delta^{>0}$.  It follows that $\gamma^{=0}\wedge \delta^{>0}$ is $\Sigma_H$-consistent.
\end{proof}

\begin{exa} Here we take as illustration the simplest case: $q = 2, d=1$ and $w = 1$.  That is to say, there is only one propositional letter and we denote it by $p$.  When depth =0, there are only two maximally consistent sets of formulas: $\{p\}$ and $\{\neg p\}$. It is easy to see that if we exclude propositional letters,  there are totally 5 maximally consistent sets of formulas:
    $\{L_0p, M_0p\}, \{L_0p, \neg M_0 p, M_{1/2}p, \neg L_{1/2} p\},
\{L_{1/2}p, M_{1/2}p\},\{L_{1/2}p, \neg M_{1/2} p, M_1p, \neg L_1p\}$ and $\{L_1p, M_1p\}$.

By Theorem \ref{Independence}, we know that $|\Omega_H(2, 1, 1)|=10$.  According to the Unique Extension Theorem, $|\Omega_H(2, d, 1)|=10$ for all $d\geq 1$.

\end{exa}

From this example, it follows immediately that
\begin{lem} Let $w=1$, i.e., the language has only one propositional letter. $|\Omega_H(q, d, 1)| = 2(2q+1) $ for all $d\geq 1$.

\end{lem}

\begin{thm} If the language has finitely many propositional letters, then $|\Omega_H(q, d, w)| $ is finite and
$|\Omega_H(q, d, w)|= |\Omega_H(q, d+k, w)|$ for all $d\geq 1$ and $k\geq 0$.
\end{thm}

\begin{proof} If the language has finitely many propositional letters, i.e., $w$ is finite, then there are finitely many maximally consistent sets of formulas of depth 0 and hence finitely many probability specifications of these maximal consistent sets in the finite language $\mathcal{L}(q, 1, w)$.  So, it follows from Theorem \ref{Independence} that, up to logical equivalence, $|\Omega_H(q, 1, w)|$ is determined by the number of Boolean combinations of formulas of depth 0 and formulas of the form $L_r \varphi$ where $\varphi$ is a formula of depth 0, and hence is finite. By Theorem \ref{UniqueExtension}, we know that $|\Omega_H(q, d, w)|= |\Omega_H(q, d+k, w)|$ for all $d\geq 1$ and $k\geq 0$.
\end{proof}

\begin{thm} \label{CardinalityOneagent} If the probability indices of the language $\mathcal{L}$ are restricted to a finite set of rationals and the language has finitely many propositional letters, i.e., $w$ is finite, then the canonical space for $\Sigma_H$ is finite.
\end{thm}


One may compare this theorem to a similar result in \cite{Aum99a} for knowledge spaces obtained through denesting in $S_5$. In the remainder of this section, we apply the notion of bi-sequence space in probability theory to the information structure in \cite{HartHS96} and construct a continuum of different $J^r$-lists (which will be defined shortly) which are all consistent. So, if the probability indices of the language are restricted within a finite set of rationals, then the canonical space of $\Sigma_H$ for at least two agents has the cardinality of the continuum.

\subsection{Bi-sequence Space}

Here we generalize the information structure in \cite{HartHS96} to the type-space setting through a notion called
bi-sequence.  Let $\Omega_s$ be the set of all pairs of infinite sequences of 0's and 1's with the same starting digit; that is,
\begin{center}
    $\Omega_s =\{(a_0a_1a_2\cdots, b_0b_1b_2\cdots):  a_k, b_k\in \{0,1\} \forall k\geq 0, a_0=b_0\}$
\end{center}

For all $k \geq 0$ let $[a_k=1]$ be the set of states whose $a_k$th coordinate is 1. Define similarly the sets $[a_k=0], [b_k=0]$ and $[b_k=1]$. Note that $[a_0=0]=[b_0=0]$ and $[a_0=1]=[b_0=1]$. Let $\mathcal{B}_s$ be the set $\{\bigcap_{m=1}^M[b_{k_m}=i_m]\cap \bigcap_{n=1}^N[a_{l_n}=j_n]: k_m,l_n\geq 0, i_m,j_n\in \{0, 1\}\}$. In other words, $\mathcal{B}_s$ is the collection of all finite intersections of the collection $\{[a_k=0], [a_k=1], [b_k=0], [b_k=1]: k\geq 0\}$. Let $\mathcal{A}_{s,0}$ be the collection of all finite unions of elements in $\mathcal{B}_s$. It is easy to check that $\mathcal{A}_{s,0}$ is an algebra.
Let $\mathcal{A}_s$ be the $\sigma$-algebra generated by $\mathcal{A}_{s,0}$.  $\mathcal{A}_s$ is simply the product $\sigma$-algebra on $\{0,1\}^{\mathbb{N}}$ restricted to $\Omega_s$.  Such defined measurable space $\langle \Omega_s, \mathcal{A}_s\rangle$ is similar to  sequence space (Section 2 in \cite{Bil95}), and is called a \emph{bi-sequence space}.

Now we define an equivalence relation $\thicksim_1$ as follows:
\begin{center}
    $(a_0a_1a_2\cdots, b_0b_1b_2\cdots)\thicksim_1(a_0'a_1'a_2'\cdots, b_0'b_1'b_2'\cdots)$ if, for all $k \geq 1$, $(a_k=a_k')$, and $( a_k=1\Rightarrow b_{k-1}=b_{k-1}')$.
\end{center}
 Let $\Pi_1$ be the partition of $\Omega_s$ into equivalence classes of $\thicksim_1$. Define $\thicksim_2$ and $\Pi_2$ analogously by interchanging the roles of $a'$s and $b'$s. So, for any $w=(a_0a_1a_2\cdots, b_0b_1b_2\cdots)\in \Omega_s$, $\Pi_1(w)$ is the set of all states which is $\thicksim_1$
 -equivalent to $w$.

\begin{defi} For any subset $D$ of $\mathbb{N}$ and any subsequence $((a_i)_{i\geq 1},(b_j)_{j\in \mathbb{N}\setminus D})$, define

\begin{center}
    $\Omega_s(((c_i)_{i\geq 1},(d_j)_{j\in \mathbb{N}\setminus D})):=\{(a_0a_1a_2\cdots, b_0b_1b_2\cdots): a_i = c_i \text{ for all } i\geq 0,  b_j=d_j \text{ for all } j\in \mathbb{N}\setminus D\}\cap\Omega_s$.
\end{center}
  In other words, $((a_i)_{i\geq 1},(b_j)_{j\in \mathbb{N}\setminus D})$ is the subsequence which is shared by all elements in $\Omega_s(((c_i)_{i\geq 1},(d_j)_{j\in \mathbb{N}\setminus D}))$.
A \emph{reduction mapping} $r_1$ from $\Omega((c_i)_{i\geq 1},(d_j)_{j\in \mathbb{N}\setminus D})$ to $\{(b_k)_{k\in D}: b_k\in \{0, 1\}\}$ is defined by
\begin{center}
    $r_1:(a_0a_1a_2\cdots, b_0b_1b_2\cdots)\mapsto (b_k)_{k\in D}$.
\end{center}
It is easy to check that this function is one-to-one.
\end{defi}

 \begin{exa} Let $w=(a_0a_1a_2\cdots, b_0b_1b_2\cdots)$ where $a_k =0$ for all odd $k$ and all other coordinates are 1. Simply $w=(101010\cdots,11111\cdots)$. Then
\begin{center}
    $\Pi_1(w)=\{(a_0'a_1'a_2'\cdots, b_0'b_1'b_2'\cdots): \forall k\geq 0(a_k'=a_k)\&
   \forall n\geq 0( b_{2n}=1)\}$
\end{center}

So an element of $\Pi_1(w)$ looks like
\begin{center}
    $(101010\cdots,1b_1'1b_3'1b_5'\cdots)$ where $b_1', b_3', b_5' \in \{0, 1\}$.
\end{center}

 There is an obvious one-to-one map between $\Pi_1(w)$ and \emph{its reduct} $r_1(\Pi_1(w)):=\{(b_1'b_3'b_5'\cdots b_{2k+1}'\cdots): k\geq 0, b_{2k+1}'\in \{0, 1\}\}$:
 \begin{center}
    $(101010\cdots,1b_1'1b_3'1b_5'\cdots)\mapsto (b_1'b_3'b_5'\cdots b_{2k+1}'\cdots)$.
 \end{center}
Since $r_1(\Pi_1(w))$ has the cardinality of the continuum, and so does $\Pi_1(w)$.

\end{exa}

For any $w\in \Omega_s$, $\Pi_1(w)$ is a subset of $\Omega_s$ of the form
    $\Omega_s((a_i)_{i\geq 1}, (b_j)_{j\in \textrm{N}\setminus D})$
for some $D\subseteq \mathbb{N}$.  Moreover, $r_1(\Pi_1(w))=\{(b_k)_{k\in K}: b_k\in \{0, 1\}\}$ is a special case of the well-known sequence space.

 \begin{lem} For any $w=(a_0a_1a_2\cdots, b_0b_1b_2\cdots)\in \Omega_s$,
    \begin{enumerate}
        \item $\Pi_1(w)$ is either finite or uncountable;
        \item $\Pi_1(w)\in \mathcal{A}_s$.
    \end{enumerate}
 \end{lem}
\begin{proof} For any $w=(a_0a_1a_2\cdots, b_0b_1b_2\cdots)\in \Omega_s$,
$\Pi_1(w)$ is the intersection of countably many events of $[a_k=i]$ or $[b_k=j]$ where $k\geq 0$ and $i, j\in \{0, 1\}$.

\begin{enumerate}
    \item Let $N_0(w):=\{k: $ the $a_k$th coordinate of $w$ is 0$\}$. It is easy to see that
        \begin{itemize}
            \item If $N_0(w)$ is countably infinite, then $\Pi_1(w)$ has the cardinality of the continuum.
            \item If $N_0(w)$ is finite, then $\Pi_1(w)$ is finite.
        \end{itemize}

    \item Since $\Pi_1(w)$ is the countable intersection of events $[a_k =i]$ or $[b_k=j]$ in $\mathcal{A}_S$, $\Pi_1(w)$ is $\mathcal{A}_S$-measurable.\qedhere
\end{enumerate}
\end{proof}

\begin{exa}\label{SeqSpace} (Sequence space) Let $S=\{0,1\}^{\infty}$ be the space of all infinite
sequences
\begin{center}
    $w=(z_0(w), z_1(w), z_2(w), \cdots)$
\end{center}
where $z_k$ is the $k$-th coordinate function mapping $S$ to $\{0, 1\}$.  A \emph{cylinder of rank n} is a set of the form
\begin{center}
    $A=\{w: (z_0(w), z_1(w), \cdots, z_n(w))\in H)\}$
\end{center}
where $H \subseteq \{0, 1\}^n$.  Let $\mathcal{C}_0$ be the set of cylinders of all ranks. It is easy to check that $\mathcal{C}_0$ is an algebra. Let $p_0=1/2$ and $p_1=1/2$.  For a  cylinder $A$ given above, define

\begin{center}
    $P(A)= \sum_H p_{u_0}p_{u_1}\cdots p_{u_n}$
\end{center}
the sum extending over all the sequences $(u_0, u_1, \cdots, u_n)$ in $H$.

Such a defined $P$ is a probability measure on the algebra $\mathcal{C}_0$ and moreover it can be uniquely extended to a probability measure on the $\sigma$-algebra $\mathcal{C}$ generated by $\mathcal{C}_0$. The interested reader may refer to Section 2 in \cite{Bil95} for more details about sequence space.

\end{exa}

In order to make the measurable space
$\langle \Omega_s, \mathcal{A}_s\rangle$ a Harsanyi type space, we will define a type function on it in three steps:
\begin{itemize}
    \item Step 1: for each state $w$ and its equivalence class $\Pi_1(w)$ of the form $\Omega_s((a_i)_{i\geq 1}, (b_j)_{j\in \textrm{N}\setminus D})$ for some $D \subseteq \mathbb{N}$, first we define a probability measure $P_{\Pi_1(w)}$ on its reduct $r_1(\Pi_1(w))$, which is finite or  a sequence space as in the above example.
    \item Step 2: according to the probability measure on $r_1(\Pi_1(w))$, we next \emph{derive} a probability measure $P_{r(\Pi_1(w))}$ on the equivalence class $\Pi_1(w)$.
    \item Step 3: finally we define a type function $T_1$ for agent 1 on the measurable space $\langle \Omega_s, \mathcal{A}_s\rangle$ such that $\langle \Omega_s, \mathcal{A}_s, T_1\rangle$ is Harsanyi type space.

\end{itemize}

   \noindent If $D$ is finite, then $\Pi_1(w)$ is finite and we define the probability measure $P_{\Pi_1(w)}$ on
    the measurable space $\langle \Pi_1(w), 2^{\Pi_1(w)}\rangle$ to be the uniform distribution on $\Pi_1(w)$. In other words, for any $w'\in \Pi_1(w)$, $P_{\Pi_1(w)}(\{w'\})=\frac{1}{|\Pi_1(w)|}$.

   If $D$ has the cardinality of the continuum, then $r_1(\Pi_1(w))$ is a sequence space with a $\sigma$-algebra $\mathcal{C}_{\Pi_1(w)}$  and a probability measure $P_{r_1(\Pi_1(w))}$ as in the above Example \ref{SeqSpace}.  Since $r_1$ is one to-one, it is easy to check that $r_1^{-1}(\mathcal{C}_{\Pi_1(w)})$ is a $\sigma$-algebra on $\Pi_1(w)$.  From $P_{r_1(\Pi_1(w))}$ on the sequence space $r_1(\Pi_1(w))$, we derive  the probability measure $P_{\Pi_1(w)}$ on the measurable space $\langle \Pi_1(w), r_1^{-1}(\mathcal{C}_{\Pi_1(w)})\rangle$ by
        \begin{center}
            $P_{\Pi_1(w)}(A):=P_{r_1(\Pi_1(w))}(r_1(A))$ for any $A\in r_1^{-1}(\mathcal{C}_{\Pi_1(w)})$.
        \end{center}

So we have finished the first two steps. Now we start the third step by defining the type function. For any $w\in \Omega_s$ and $E\in \mathcal{A}_s$, define
\begin{center}
    $T_1(w)(E): = P_{\Pi_1(w)}(E\cap \Pi_1(w))$
\end{center}
It is easy to see that $T_1(w)(\Pi_1(w))=1$.  In some sense, $T_1(w)$ is a probability measure on $\langle \Omega_s, \mathcal{A}_s\rangle$ concentrating on the equivalence class $\Pi_1(w)$.

\begin{prop}\label{TypeFunction} Such a defined $T_1$ is a type function on the measurable space $\langle \Omega_s, \mathcal{A}_s\rangle$.

\end{prop}

In order to show the proposition, we need to prove several lemmas. Recall that $B_1^{\geq r}(E)$ denotes the \emph{set} $\{w\in \Omega_s: T_1(w, E)\geq r\}$ and $B_1^{=r}(E)$ the \emph{set} $\{w\in \Omega_s: T_1(w, E)=r\}$. Let $\alpha_k$ and $\beta_k(k\in \textrm{N})$ be the two \emph{coordinate functions} from $\Omega_s$ to $\{0, 1\}$:

\begin{center} $\alpha_k: ((a_i)_{i\in \textrm{N}},(b_j)_{j\in \textrm{N}})\mapsto a_k$ and
$\beta_k: ((a_i)_{i\in \textrm{N}},(b_j)_{j\in \textrm{N}})\mapsto b_k$.
\end{center}

\begin{prop}  For any event $E=\bigcap_{m=1}^M[b_{k_m}=i_m]\cap \bigcap_{n=1}^N[a_{l_n}=j_n](k_m,l_n\geq 0, i_m,j_n\in \{0, 1\})\in \mathcal{B}_s$, $B^{\geq r}_1(E)\in \mathcal{A}_s$ for all $r\in [0, 1]$, or equivalently, the function $T_1(\cdot, E)$ is $\mathcal{A}_s$-measurable.

\end{prop}

\begin{proof}  Here we take $E=[b_1=0]\cap [a_1=1]$  as an illustration. The proof of other cases is similar.
Given a state $w\in \Omega_s$, we divide the proof into the following cases:
    \begin{itemize}
        \item Case 1: $\alpha_1(w)=0$.  In this case $T_1(w, E)=0$
        \item Case 2: $\alpha_1(w)=1, \alpha_2(w)=1$ and $\beta_1(w)=0$. In this case, $T_1(w, E)=1$.
        \item Case 3: $\alpha_1(w)=1, \alpha_2(w)=1$ and $\beta_1(w)=1$. In this case, $T_1(w, E)=0$
        \item Case 4: $\alpha_1(w)=1, \alpha_2(w)=0$. In this case, $T_1(w, E)=1/2$.
    \end{itemize}
    So $T_1(\cdot, E)$ is an $\mathcal{A}_s$-measurable function.
\end{proof}

\begin{thm}\label{Algebra}  For any event $E\in \mathcal{A}_{s,0}$, $B^{\geq r}_1(E)\in \mathcal{A}_s$ for all $r\in [0, 1]$, or equivalently, the function $T(\cdot, E)$ is $\mathcal{A}_s$-measurable.
\end{thm}

\begin{proof} Given an event $E\in \mathcal{A}_{s,0}$, we know that it is a finite \emph{disjoint} union of events in $\mathcal{B}_s$. That is to say, $E = \uplus_{i=1}^n E_i $ for some events $ E_i\in \mathcal{B}_s(i=1, \cdots, n)$. It follows that, for any $w\in \Omega_s$,
\begin{center}
    $T_1(w, E)= \sum_{i=1}^n T_1(w, E_i)$
\end{center}
 Since all $T_1(\cdot, E_i)$'s are $\mathcal{A}_s$-measurable, so is $T_1(\cdot, E)$.
\end{proof}

Proof of Proposition \ref{TypeFunction}:  In Theorem \ref{Algebra}, we have shown that, for any $E\in \mathcal{A}_{s,0}$, $T_1(\cdot, E)$ is an $\mathcal{A}_s$-measurable function. By Lemma \ref{MainLemma} , we know that, for any $E\in \mathcal{A}_s$, $T_1(\cdot, E)$ is also an $\mathcal{A}_s$-measurable function. So indeed $T_1$ is a type function.

\begin{thm} $\langle \Omega_s, \mathcal{A}_s, T_1\rangle$ is a Harsanyi type space.
\end{thm}

\begin{proof} This follows directly from the definition of $T_1$.
\end{proof}

Note that all the corresponding notions $r_2, T_2$
 and $B_2^{\geq r}$ for agent 2 can be defined dually by interchanging the roles of a's and b's.  And the corresponding propositions also hold for agent 2.

\subsection{$J^r$-lists}

Let $X$ be the event $\{(a_0a_1\cdots, b_0b_1\cdots)\in \Omega_s: a_0=b_0=1\}$. For agent $i$'s, define two new  operators $J_i^{\geq r}: \mathcal{A}_s\rightarrow \mathcal{A}_s$ by
\begin{center} $J_i^rE := (B_i^{\geq r} E)\cup (B_i^{\geq r} \neg E)(i=1, 2, r\in [0, 1]\cap \mathbb{Q})$
\end{center}
Note that $J_i^rE = J_i^r \neg E$.  Let $s=(E_1, E_2, \cdots, E_n)$ be a list of events. $s_{last}$ denotes the last event $E_n$ in the list and $s_{initial}$ denotes the initial segment of $s$ excluding the last event $s_{last}$.

\begin{defi} For  $r\in [0, 1]$, a (finite)\emph{$J^r$-list} is defined inductively as follows:
    \begin{enumerate}
        \item $(X)$ and $(\neg X)$ are $J^r$-lists;
        \item If $s$ is a $J^r$-list and the operator of $s_{last}$ is $J_1^r$ or $\neg J_1^r$, then $(s, J_2^rs_{last})$ and $(s, \neg J_2^rs_{last})$ are both $J^r$-lists;
        \item If $s$ is a $J^r$-list and the operator of $s_{last}$ is $J_2^r$ or $\neg J_2^r$, then $(s, J_1^rs_{last})$ and $(s, \neg J_1^rs_{last})$ are both $J^r$-lists;

    \end{enumerate}
\end{defi}
\newpage
According to the definition $(X)$ is a $J^r$-list. Applying both $J_1^r$ and $\neg J_2^r$ to $X$, we have
\begin{center} $(X, J_1^r X), (X, \neg J_2^r X)$
\end{center}
Both of them are $J^r$-lists.  Next adding $\neg J_2^r$ and $J_1^r$ to the last events of the lists, respectively, we have
\begin{center} $(X, J_1^r X,\neg J_2^rJ_1^r X ), (X, \neg J_2^r X,J_1^r\neg J_2^r X )$.
\end{center}
Both of them are $J^r$-lists.  We can go on, ad infinitum, generating infinitely many \emph{infinite} $J^r$-lists. It is easy to see that  the set $S_{\infty}$ of infinite $J^r$-lists has the cardinality of the continuum.   A (finite or infinite) $J^r$-list is \emph{consistent}\footnote{This notion id different from the consistency of formulas that we define in Section 2.} if the intersection of all the events in the list is nonempty.

Since $J_i^rE=J_i^r\neg E$ for any event $E$, the above lists are also the same as
\begin{center} $(X, J_1^r X,\neg J_2^rJ_1^r X ), (X, \neg J_2^r X, J_1^r J_2^r X )$.
\end{center}
That is to say, we can omit all complementation signs \emph{inside} an event.
 When $r\leq 1/2$,  the event $\neg J_i^r E = \neg B_i^{\geq r} E \cap \neg B_i^{\geq r} \neg E$ is an empty set. This implies that when $r\leq 1/2$, all \emph{negated} $J_i^r$-events are not consistent and hence not all $J^r$-lists are consistent.

However, when $r>1/2$, $J_i^r$ is called a ``strongly believing whether" operator for agent $i$ and we show that all $J^r$-lists are consistent.

\begin{lem} Let $r >1/2$.  For all $k\geq 0$,
    \begin{itemize}
        \item $[a_k=1]= \underbrace{J_1^rJ_2^rJ_1^r\cdots}_{k} X, [a_k=0]= \neg \underbrace{J_1^rJ_2^rJ_1^r\cdots}_{k} X$
        \item $[b_k=1]= \underbrace{J_2^rJ_1^rJ_2^r\cdots}_{k} X, [b_k=0]= \neg \underbrace{J_2^rJ_1^rJ_2^r\cdots}_{k} X$
    \end{itemize}
\end{lem}
\begin{proof} We prove this lemma by induction on $k$. When $k=0$, it is easy to see that the lemma holds by the definition of $X$.  Note that $[a_k=1]=\neg [a_k=0]$ and $[b_k=1]=\neg [b_k=0]$.  For $k\geq 1$, it suffices to show the following inductive relations:
\begin{itemize}
    \item $[a_k=1]= J_1^r [b_{k-1}=1]$
    \item $[b_k=1]=J_2^r [a_{k-1}=1]$.
\end{itemize}

Here we prove the first equality.  That is say, we show
\begin{center}
    $[a_k=1]= (B_1^{\geq r} [b_{k-1}=1])\cup (B_1^{\geq r} [b_{k-1}=0])$
\end{center}

Assume that $w\in [a_k=1]$. That is to say, $\alpha_k(w) =1$. According to the definition of $\thicksim_1$, the $b_{k-1}$th coordinates of all elements in the equivalence class $\Pi_1(w)$ are the same as $\beta_{k-1}(w)$, the $b_{k-1}$th coordinate of $w$. It follows that $\Pi_1(w)\subseteq [b_{k-1}=\beta_{k-1}(w)]$ and hence $T_1(w, [b_{k-1}=\beta_{k-1}(w)])=1$.  So $w\in(B_1^{\geq r} [b_{k-1}=1])\cup (B_1^{\geq r} [b_{k-1}=0])$.

Now we show the other inclusion. Assume that  $w\not\in [a_k=1]$.  It follows that one half elements in the equivalence class $\Pi_1(w)$ have 0 as their $b_{k-1}$th coordinates and the other half have 1 as their $b_{k-1}$th coordinates.  By the definition of $T_1$, we know that
\begin{center} $T_1(w, [b_{k-1}=0])= 1/2$  and $T_1(w, [b_{k-1}=1])=1/2$.
\end{center}
Therefore, $w\not \in (B_1^{\geq r} [b_{k-1}=1])$ and $w\not \in (B_1^{\geq r} [b_{k-1}=0])$.  So $w\not \in ((B_1^{\geq r} [b_{k-1}=1])\cup (B_1^{\geq r} [b_{k-1}=0]))$.   The second equality  can be shown similarly. So
we have proved the lemma.
\end{proof}

\begin{thm}  For $r>1/2$, all $J^r$-lists are consistent. In particular, all $J^1$-lists are consistent.
\end{thm}

\begin{proof} This theorem follows directly from the above lemma.  Consider the $J^r$-list \begin{center}
 $(X, \neg J_2^r X,J_1^rJ_2^r X, \neg J_2^r J_1^rJ_2^r X )$
\end{center}
By the above lemma, $X=[a_0=1], \neg J_2^r X = [b_1=0],J_1^rJ_2^r X=[a_2=1] $, and $\neg J_2^r J_1^rJ_2^r X =[b_3=0]$.  That is to say,
the intersection of all the events in this list is  $[a_0=1]\cap [b_1=0]
\cap   [a_2=1]\cap [b_3=0] $. So this $J^r$-list is consistent. From this example, we can easily generalize that every infinite $J^r$-list is the intersection of $\{[a_i = k_i]: i\geq 0\}$ and $\{[b_j=l_j]:j \geq 0\})$ for some $k_i, l_j\in \{0,1\}$, which includes the list $(a_0a_1a_2\cdots, b_0b_1b_2\cdots)$ as an element, and is hence consistent.
 \end{proof}

\begin{thm} Assume that the probability indices of the language $\mathcal{L}$ are restricted within a finite set of rationals. If further the language contains only one propositional letter $p$, then the canonical space of $\Sigma_H$ for two agents
has the cardinality of the continuum.
\end{thm}

\begin{proof} Assume that the probability indices of the language $\mathcal{L}$ are restricted within a finite set of rationals and the language contains only one propositional letter $p$. If we replace $X$ in each $J^1$-list by a proposition letter $p$ and $B^{\geq r}_i$ by $L_r^i$, we obtain a list of formulas in this restricted language which is satisfied in the Harsanyi type space
 $\langle \Omega_s, \mathcal{A}_s, T_i\rangle$  by simply interpreting them as follows:

    \begin{itemize}
    \item $V(p):= X$;
    \item for any formula $\phi$ in the restricted language, $[[L_r^i\phi]]:= B_i^{\geq r}[[\phi]]$

    \end{itemize}
 Consider the set $S^1_{\infty}$ of infinite such $J^1$-lists of formulas starting with $J^1_1$ operator or its negation:
 \begin{center}
    $S^1_{\infty}=\{(\pi_0p, \pi_1 J_1^1p, \pi_2 J_2^1J_1^1p,\cdots  ): \pi_k$ is either blank or $\neg$ $(k \geq 0)\}$
 \end{center}

  It is easy to see that the set of formulas in each list in $S^1_{\infty}$ is satisfied in the Harsanyi type space
 $\langle \Omega_s, \mathcal{A}_s, (T_i)_{i=1,2}\rangle$ but the set of formulas in any two lists is not. Since $S_{\infty}^1$ has the cardinality of the continuum, so does the canonical space of $\Sigma_H$
 for two agents.
\end{proof}

\begin{thm} \label{CardinalityMultiagent}If the probability indices of the language $\mathcal{L}$ are restricted within a finite set of rationals and there are finitely many propositional letters in the language, then the canonical space $\Omega_H$ of $\Sigma_H$
for $n$ agents($2\leq n < \aleph_0$) has the cardinality of the continuum.
\end{thm}

\begin{proof} This proposition follows from the above theorem  with the following two observations:
    \begin{itemize}
        \item $\Omega_H$ has the an upper bound $2^{\aleph_0}$, since the set of formulas in the restricted language is countable and hence its power set has the cardinality of the continuum;
        \item the cardinality of the canonical space does not decrease with the increase of  the number of agents or of the number of the propositional letters.\qedhere
    \end{itemize}
\end{proof}

\noindent The combination of  Theorems  \ref{CardinalityOneagent} and \ref{CardinalityMultiagent} justifies the relative complexity of the multi-agent interactive epistemology compared with one-agent epistemology from the perspective of probabilistic belief.

 \section{Relationship between knowledge and probabilistic beliefs}

  It is well-known in interactive epistemology \cite{Aum99b, AH02} that the concept of knowledge is \emph{implicit in probabilistic beliefs} in the above semantic framework.  The logic for knowledge is the well-known $S_5$, which is the smallest normal modal logic plus the following axiom schema \cite{BRV00}:
  \begin{itemize}
    \item ($T_K$) $K_i p \rightarrow p$
    \item ($4_K$) $K_i p \rightarrow K_iK_i p$
    \item ($5_K$) $\neg K_i p \rightarrow K_i\neg K_i p$.
  \end{itemize}

   \noindent Recall that, given a Harsanyi type space $\langle \Omega, \mathcal{A}, (T_i)_{i=1,2}\rangle$, for each state $w\in \Omega$ and agent $i$, $[T_i(w)]$ denotes the set of all states in which $i's$ probabilities measures are the same as in $w$, i.e., $[T_i(w)]: =\{s'\in S: T_i(s') (A)= T_i(s)(A)$ for all $A\in \mathcal{A}$ $\}$.  $[T_i(w)]$ induces a partition of $\Omega$ and hence a knowledge operator $K_i$, which is an operator on $\mathcal{A}$ satisfying the well-known $S_5$ axioms. It is easy to check that $K_i$ satisfies the following properties: for any $E\in \mathcal{A}$,

  \begin{enumerate}
    \item $K_i E \subseteq B_i^{\geq 1} E$;
    \item $B_i^{\geq r} E \subseteq K_i (B_i^{\geq r} E)$;
    \item $(\Omega \setminus B_i^{\geq r} E) \subseteq K_i (\Omega \setminus B_i^{\geq r} E)$;
  \end{enumerate}

  \noindent Let $\Pi_i$ be another partition and $K'_i$ be its associated knowledge operator satisfying the above three properties.  It is shown in \cite{AH02} that, for each $w\in \Omega$, $[T_i(w)] = \Pi_i(w)$.

  However,  syntactically knowledge is a separate and exogenous notion which is not redundant. Indeed, the probability syntax can only express beliefs of the agents, not that an agent knows anything about another one for sure.  For the completeness of the presentation, here we repeat the example in Section 8 in \cite{AH02}.  Let $H$ and $H'$ be two two-agent Harsanyi type spaces. In $H$, there is only one state $w$, both Ann and Bob assign probability 1 to $w$. In $H'$, there are two states $w$ and $w'$.  At $w$, both Ann and Bob assign probability 1 to $w$; at $w'$, Ann assigns probability 1 to $w$ and Bob assigns probability 1 to $w'$.  In addition, the proposition letter $p$ is true only at $w$ in both spaces. It is easy to see that in $H$, Ann knows for sure in $w$ that Bob assigns probability 1 to $p$ whereas she does not in $H'$.  Moreover, there is no way to capture this syntactically without explicitly introducing knowledge operators for Ann and Bob.

 In the following, we give a new \emph{logical} characterization of this relationship  between  probabilistic belief and knowledge by generalizing the three notions of definability in multi-modal logics \cite{HSS09a, HSS09b} to the setting of probabilistic beliefs and knowledge to show that this relationship is equivalent to the statement that knowledge is \emph{implicitly defined by probabilistic belief}, but is not \emph{reducible to it} and hence is not \emph{explicitly definable} in terms of probabilistic belief.  First we briefly review the background theory about definability in multi-modal logic \cite{HSS09a} and then apply it to the relationship between knowledge and
 probabilistic belief.

  A language $\mathcal{L}((O_i)_{i\in \textrm{N}})$ is defined as follows. We start from a countable set of propositional letters $AP:=\{p_0, p_1, \cdots\}$.  The set of  \emph{formulas} in the language $\mathcal{L}((O_i)_{i\in \mathbb{N}})$ is built from propositional letters by connectives $\neg$ and $\wedge$, and the set of operators $(O_i)_{i\in \mathbb{N}}$.  Equivalently, a formula is defined according to the following syntax:
 \begin{center}
    $\phi:= p \mid \neg \phi \mid \phi\wedge \phi \mid O_i\phi (i\in \mathbb{N})$
 \end{center}

 A \emph{logic}  $\Lambda$ in the language $\mathcal{L}((O_i)_{i\in \textrm{N}})$ is a set of formulas that contains all propositional tautologies, is closed under modus ponens and uniform substitution (that is, if $\phi\in \Lambda$, do do all its substitution instances.).   Let $\Lambda$ and $\Lambda'$ be two logics. $\Lambda+\Lambda'$ denotes
 the smallest logic containing $\Lambda \cup \Lambda'$.  Similarly, for a formula $\phi$, $\Lambda +\phi$
 denotes the smallest logic that contains $\Lambda\cup\{\phi\}$.    Let $O'$ be an operator new to
 $\Lambda$. $\Lambda[O_i/O']$  is the logic obtained by replacing $O'$ for all occurrences
  $O_i$ in all formulas in $\Lambda$.

 $\Lambda$ is \emph{normal} if, in addition, for each operator $O_i$,
\begin{itemize}
    \item $K_{O_i}: O_i(\phi\rightarrow \psi)\rightarrow (O_i\phi \rightarrow O_i\psi)\in \Lambda$;
    \item it is closed under generalization: $(\phi\in \Lambda\Rightarrow O_i\phi\in \Lambda)$.
\end{itemize}
Note that we don't require logics to be normal in this paper, since, obviously, probability operators are not necessarily normal.

\subsection{Algebras and Algebraic models}

    An \emph{algebra} $\mathcal{A}$ for the language $\mathcal{L}((O_i)_{i\in \mathbb{N}})$ is a tuple
    \begin{center}
        $\mathcal{A}:=\langle \mathcal{B}, \neg, \wedge, 1, (O_i)_{i\in \mathbb{N}}\rangle $
    \end{center}
    where $\langle \mathcal{B}, \neg, \wedge, 1\rangle$ is a Boolean algebra and $O_i(i\in \mathbb{N})$ is a unary operator on $\mathcal{B}$.  An \emph{algebraic model} $\mathcal{M}$ based on this algebra $\mathcal{A}$ is a pair $\langle \mathcal{A}, v\rangle$ where $v$ is a \emph{valuation function} from the set $AP$ of propositional letters to $\mathcal{B}$.
As usual, $v$ can be extended to a meaning function $[[\cdot]]_{\mathcal{M}}$ on all formulas inductively as follows:

\begin{itemize}
    \item $[[p]]_{\mathcal{M}} = v(p)$ for all proposition letters;
    \item $[[\phi\wedge \psi]]_{\mathcal{M}}= [[\phi]]_{\mathcal{M}}\wedge_{\mathcal{M}} [[\psi]]_{\mathcal{M}}$;
    \item $[[\neg \phi]]_{\mathcal{M}}=\neg_{\mathcal{M}} [[\phi]]_{\mathcal{M}}$;
    \item $[[O_i\phi]]_{\mathcal{M}}= O_i([[\phi]]_{\mathcal{M}})$.

\end{itemize}

We omit the subscript $\mathcal{M}$ if no confusion arises. A formula $\phi$ is \emph{valid in $\mathcal{M}$} if $[[\phi]]=1$ and is \emph{valid in the underlying algebra $\mathcal{A}$} if $\phi$ is valid in all algebraic models based on $\mathcal{A}$.  It is \emph{valid in a class
of algebras} if it is valid in all algebras in this class. Let $Th(\mathcal{M})$ be the set of all formulas which are valid in $\mathcal{M}$ and $Th(\mathcal{A})$ be the set of formulas which are valid in $\mathcal{A}$. A logic $\Lambda$ is \emph{sound} with respect to $\mathcal{M}$ ($\mathcal{A}$) if $\Lambda \subseteq Th(\mathcal{M})(Th(\mathcal{A}))$; also we say, $\mathcal{A} (\mathcal{M})$ is a \emph{$\Lambda$-algebra }(\emph{an algebraic model for $\Lambda$}). And it is \emph{complete} with respect to  $\mathcal{M}$ ($\mathcal{A}$) if we reverse the above inclusion, i.e., $\Lambda \supseteq Th(\mathcal{M})(Th(\mathcal{A}))$.

For a logic $\Lambda$ in the language $\mathcal{L}((O_i)_{i\in \textrm{N}})$, we define an equivalence relation $\equiv_{\Lambda}$  on formulas by $\phi\equiv_{\Lambda}\psi $ if
$\phi \leftrightarrow \psi\in \Lambda$. $|\phi|_{\Lambda}$ denotes the equivalence class that contains
$\phi$. Let $\mathcal{L}((O_i)_{i\in \textrm{N}})/{\equiv_{\Lambda}}$ be the collection of all equivalent classes. The corresponding algebraic operations are defined as follows:
\begin{itemize}
        \item $\neg |\phi|_{\Lambda} = |\neg \phi|_{\Lambda} $
        \item $|\phi|_{\Lambda}\wedge |\psi|_{\Lambda}= |\phi\wedge \psi|_{\Lambda}$
        \item $O_i(|\phi|_{\Lambda})= |O_i\phi|_{\Lambda}$
\end{itemize}
 
It is easy to check that these operations are well-defined and such defined
 $\mathcal{A}_{\Lambda}:=\langle \mathcal{L}((O_i)_{i\in \textrm{N}})/{\equiv_{\Lambda}},$ $
 \neg, \wedge, (O_i)_{i\in \textrm{N}}\rangle$ is an algebra. We call it \emph{the Lindenbaum algebra for the logic } $\Lambda$.  The \emph{canonical algebraic model $\mathcal{M}_{\Lambda}$} is a tuple:
 \begin{center}
         $\langle \mathcal{L}((O_i)_{i\in \textrm{N}})/{\equiv_{\Lambda}},\neg, \wedge, (O_i)_{i\in \textrm{N}}, v_{\Lambda} \rangle $
 \end{center}
 where $\langle \mathcal{L}((O_i)_{i\in \textrm{N}})/{\equiv_{\Lambda}},
 \neg, \wedge,(O_i)_{i\in \textrm{N}}\rangle $  is the Lindenbaum algebra for $\Lambda$ and $v_{\Lambda}$ is the canonical valuation: $v_{\Lambda}(p)=|p|_{\Lambda}$.

 \begin{prop}   A logic $\Lambda$ in the language $\mathcal{L}((O_i)_{i\in \textrm{N}})$ is sound and complete with respect to $\{\mathcal{M}_{\Lambda}\}$ and
 hence with respect to $\{\mathcal{A}_{\Lambda}\}$. Therefore, it is sound and complete with
 respect to the class of $\Lambda$-algebras.

 \end{prop}
 \begin{proof} The interested reader may refer to Chapter 5 in \cite{BRV00} for details.
 \end{proof}

 \subsection{Three Notions of Definability}
  Let $\delta$ be a formula in the language $\mathcal{L}((O_i)_{i\in \mathbb{N}})$. $O_{\infty}$ is new to $\mathcal{L}((O_i)_{i\in \mathbb{N}})$.  We call the
  following formula  a \emph{definition of $O_{\infty}$} in terms of operators $(O_i)_{i\in \mathbb{N}}$
  \begin{center} $DO_{\infty}: O_{\infty}p\leftrightarrow \delta$,
  \end{center}
  \noindent which is a formula in $\mathcal{L}((O_i)_{i\in \mathbb{N}},
  O_{\infty})$. Let $\Lambda$ be a logic in $\mathcal{L}((O_i)_{i\in \mathbb{N}}, O_{\infty})$.  For
  a smaller language $\mathcal{L}'\subseteq \mathcal{L}$, if $\Lambda'= \Lambda \cap \mathcal{L}'$,
  we say $\Lambda$ is a \emph{conservative extension} of $\Lambda'$.    $\Lambda_0$
  denotes $\Lambda\cap \mathcal{L}((O_i)_{i\in \mathbb{N}})$, which is a logic.
  If $\mathcal{A}=\langle \mathcal{B}, \neg, \wedge, 1, (O_i)_{i\in \mathbb{N}}, O_{\infty}\rangle$
  is a $\Lambda$-algebra, then it is easy to show that $\langle \mathcal{B}, \neg, \wedge, 1, (O_i)_{i \in \mathbb{N}}\rangle$
  is a $\Lambda_0$-algebra, which is denoted $\mathcal{A}_0$. So $\mathcal{A}$ is also written as $\langle \mathcal{A}_0, O_{\infty}\rangle$.

  \begin{defi}  $O_{\infty}$  is \emph{explicitly defined} in $\Lambda$ if there is a definition $DO_{\infty}$   such that $DO_{\infty}\in \Lambda$.    $O_{\infty}$ is
      \emph{implicitly defined} in $\Lambda$ if $O_{\infty} p\leftrightarrow O_{\infty}'p
      \in \Lambda + \Lambda[O_{\infty}/O_{\infty}']$.  $O_{\infty}$ is \emph{reducible} to
      $(O_i)_{i\in \mathbb{N}}$ if there is a definition of $O_{\infty}$ in terms of $(O_i)_{i\in \textrm{N}}$ such that $\Lambda_0+ DO_{\infty}$ is a conservation extension of $\Lambda_0$ and $\Lambda \subseteq \Lambda_0 + DO_{\infty}$.

  \end{defi}

  An algebra $\mathcal{O}$ of unary operators over a Boolean algebra $\langle \mathcal{B}, \neg, \wedge, 1\rangle$ is a Boolean algebra which is closed under composition where operations on operators are defined as follows:
   \begin{itemize}
          \item for any $f\in \mathcal{O}$, $(\neg f)(x):= \neg f(x)$;
          \item for any $f, g\in \mathcal{O}$, $(f\wedge g)(x):= f(x)\wedge g(x)$;
          \item for any $f, g\in \mathcal{O}$, $(f\circ g)(x):= f(g(x))$;
          \item $1$ is a constant operator which maps each element in $\mathcal{B}$ to 1.
   \end{itemize}
   For a $\Lambda$-algebra $\mathcal{A}=\langle \mathcal{B}, \neg, \wedge, 1, (O_i)_{i\in \mathbb{N}}, O_{\infty}\rangle$, the algebra $\mathcal{O}^*_{\mathcal{A}}$ of operators over $\mathcal{A}$ is the smallest algebra that includes $(O_i)_{i\in \mathbb{N}}$ and $ O_{\infty}$.
     Such a defined $\langle \mathcal{O}^*_{\mathcal{A}}, \neg, \wedge, \circ, 1_{\mathcal{O}^*_{\mathcal{A}}}\rangle$ is called an \emph{algebra of operators} on $\mathcal{A}$.
      $\mathcal{O}^*_{\mathcal{A}_0}$ can be defined similarly for the associated  \emph{reduced} algebra $\mathcal{A}_0$.

       The following three propositions are the algebraic characterizations of these three notions of definability.

       \begin{prop}\label{Implicit}  $O_{\infty}$ is implicitly defined in $\Lambda$ if and
       only if, for any two $\Lambda$-algebras $\langle \mathcal{A}_0, O_{\infty}\rangle$
       and $\langle \mathcal{A}_0, O'_{\infty}\rangle, O_{\infty} =O'_{\infty}$.

       \end{prop}

       \begin{prop}\label{Explicit}
        $O_{\infty}$ is explicitly defined in $\Lambda$ if and only if
       for each $\Lambda$-algebra $\mathcal{A}$, $\mathcal{O}^*_{\mathcal{A}}=
               \mathcal{O}^*_{\mathcal{A}_0}$.
       \end{prop}

       \begin{prop} \label{Reducible} $O_{\infty}$ is reducible to $(O_i)_{i\in \textrm{N}}$
       in $\Lambda$ if and only if  each $\Lambda_0$-algebra $\mathcal{A}_0$ can be
       extended to a $\Lambda$-algebra $\mathcal{A}$ such that $\mathcal{O}^*_{\mathcal{A}}=\mathcal{O}^*_{\mathcal{A}_0}$.
       \end{prop}

       \begin{prop}\label{EquivalenceDef}  $O_{\infty}$ is explicitly defined in $\Lambda$ if and only if
        $O_{\infty}$ is implicitly defined  and is reducible to $(O_i)_{i\in \mathbb{N}}$
        in $\Lambda$.
       \end{prop}

       The proofs of the above 4 propositions are similar to the corresponding lemmas in
       \cite{HSS09a, HSS09b}: the left-to-right directions are shown by the meaning functions on algebraic models and the right-to-left ones are proved through canonical models. The interested reader may refer to these papers for detailed proofs. It is worth noting that the proofs in \cite{HSS09a, HSS09b} for modal operators also apply to our probability operators $L_r$ in this paper.

\subsection{Definability of knowledge in terms of probabilistic belief}
 We have finished the review of the background theory about definability in multi-modal logic.  Now we apply it to the relationship between knowledge and probabilistic belief. In this section, the language $\mathcal{L}$ in previous sections are written as  $\mathcal{L}((L_r)_{r\in [0, 1]\cap \mathbb{Q}})$ in order to make explicit the operators $(L_r)_{r\in [0, 1]\cap \mathbb{Q}}$. Let  $\mathcal{L}_{HK}$ be the language $\mathcal{L}((L_r)_{r\in [0, 1]\cap \mathbb{Q}}, K)$.  In the remainder of this section, we mainly consider the logic $\Sigma_{HK}:=(\Sigma_H + (S_5)_K +\{H_1, H_2, H_3\}) \subseteq \mathcal{L}((L_r)_{r\in [0, 1]\cap \mathbb{Q}}, K)$, where the operators $L_r (r\in [0, 1]\cap \mathbb{Q})$ satisfy the logic $\Sigma_H$, and $K$ satisfies the well-known $S_5$ axioms
  and these operators are connected by the following 3 axioms:

\begin{itemize}
    \item $(H1): L_r \phi \rightarrow K L_r \phi$;
    \item $(H2): \neg L_r \phi \rightarrow K\neg  L_r \phi$;
    \item $(H3): K\phi \rightarrow L_1 \phi$.

\end{itemize}

\begin{defi} A \emph{knowledge-belief space} is a tuple $\langle \Omega, \mathcal{A}, T, \Pi\rangle$ where
    \begin{itemize}
        \item $\langle \Omega, \mathcal{A}, T\rangle$ is a Harsanyi type space;
        \item $\langle \Omega, \Pi\rangle$ is a knowledge space where $\Pi$ is a partition on $\Omega$;
        \item For each $w\in \Omega$ and $E\in \mathcal{A}$, $\Pi(w) \subseteq E$ implies $T(w)(E)=1$;
        \item  For each $w\in \Omega$ and $E\in \mathcal{A}$, $[T(w)] \subseteq E$ implies $\Pi(w) \subseteq E$.

    \end{itemize}
\end{defi}

\begin{thm} \label{CompletenessKnowledgeBelief} The logic $\Sigma_{HK}$ is sound and complete with respect to the class of
knowledge-belief spaces.

\end{thm}
\begin{proof} The interested reader may refer to the detailed proofs in \cite{FH94, Aum99b,Zhou07}.
\end{proof}

\begin{thm} $K$ is implicitly defined in the logic $\Sigma_{HK}$.
\end{thm}

\begin{proof} Let $\mathcal{A}_1=\langle \mathcal{B}, \neg, \wedge, 1, (B^{\geq r})_{r\in [0, 1]\cap \mathbb{Q}}, K_1\rangle$ and $\mathcal{A}_2=\langle \mathcal{B}, \neg, \wedge, 1, (B^{\geq r})_{r\in [0, 1]\cap \mathbb{Q}}, K_2\rangle$ be two $\Sigma_{HK}$-algebras, where $B^{\geq r}$ is the corresponding semantical interpretation of $L_r$, and $K_1$ and $K_2$ are the corresponding interpretations of the syntactical $K$ in $\mathcal{L}_{HK}$.  It is easy to check that the following equalities hold:

\begin{itemize}
    \item For any $e\in \mathcal{B}, B^{=1} (K_1 e) = K_1 e$;
    \item For any $e\in \mathcal{B}, B^{=1} (K_2 e) = K_2 e$;
    \item For any $e\in \mathcal{B},K_1 (B^{=1}  e) = B^{=1} e$;
    \item For any $e\in \mathcal{B},K_2 (B^{=1}  e) = B^{=1} e$;
\end{itemize}

Since $K_1 e = B^{=1} K_1 e$,
\begin{eqnarray}
K_2(K_1e) & = & K_2(B^{=1}(K_1 e)) \nonumber\\
 & = & B^{=1}(K_1 e)\nonumber \\
 &= & K_1 e\nonumber
\end{eqnarray}
Since $K_1 e \leq e$, $K_2 (K_1 e) \leq K_2 e$. So $K_1 e\leq K_2 e$.   Similarly, we can show that $K_2 e\leq K_1 e$.  So $K_1 e = K_2 e$.  That is to say, $K_1 = K_2$.
\end{proof}

The following example is adapted from Example 5.3 in \cite{HSS09a}.

 \begin{exa} Let $W=[0, 1], \mathcal{B}=\{B\subseteq [0, 1]:$ $B$ is a Borel set in
 [0, 1] and $\lambda (B) =0$ or 1 $\}$, and $\mathcal{U}=\{B\subseteq [0, 1]:$ $B$ is a Borel set in
 [0, 1] and $\lambda (B) =1\}$ where $\lambda$ is the Lebesgue measure.  It is easy to check that $\mathcal{B}$ is an algebra under the set complementation and the set intersection.  In some sense we can say that $\mathcal{U}$ is the set of all ``big" events in $[0,1]$. Define the probability operators $(B^{=r})_{r\in [0,1]\cap \mathbb{Q}}$ as follows:

 \begin{displaymath}
B^{=r} E = \left\{ \begin{array}{ll}
E\cup \{0\} & \textrm{if $E\in \mathcal{U}$ and $r=1$}\\
\emptyset & \textrm{if $E\in \mathcal{U}$ and $0<r<1$}\\
E^c\setminus \{0\} & \textrm{if $E\in \mathcal{U}$ and $r=0$}\\
E\setminus \{0\} & \textrm{if $E\not\in \mathcal{U}$ and $r=1$}\\
\emptyset & \textrm{if $E\not\in \mathcal{U}$ and $0<r<1$}\\
E^c\cup \{0\} & \textrm{if $E\not\in \mathcal{U}$ and $r=0$}\\
\end{array} \right.
\end{displaymath}

It is easy to check that $B^{=r}(r\in[0,1]\cap\mathbb{Q})$ indeed define probability operators
$B^{\geq r}(r\in [0,1]\cap \mathbb{Q})$ on the algebra $\mathcal{B}$.  Note that, if $0\leq  r\leq s\leq 1$, then $B^{\geq r} E\supseteq B^{\geq s}E$ for any $E\in \mathcal{B}$.

 \end{exa}

 \begin{lem}\label{Extension} The above defined algebra $\mathcal{A}= \langle\mathcal{B}, \neg, \cap, W, (B^{\geq r})_{r\in [0,1]\cap \mathbb{Q}}\rangle$ is a $\Sigma_H$-algebra that can not be extended to a $\Sigma_{HK}$-algebra.
 \end{lem}

The tedious \emph{verification} of this lemma is relegated to the Appendix.

\begin{thm} $K$ can not be reduced to $(L_r)_{r\in [0, 1]\cap Q}$ in the logic  $\Sigma_{HK}$;
therefore it cannot be explicitly defined in this logic.
\end{thm}
\begin{proof} This theorem follows directly from the above Lemma \ref{Extension}, Propositions \ref{Reducible} and \ref{EquivalenceDef}.

\end{proof}

\begin{thm} \label{ConservationKnowledgeBelief}Every Harsanyi type space $\langle \Omega, \mathcal{A}, T\rangle$ can be extended to a knowledge-belief system $\langle \Omega, \mathcal{A}, T, \Pi\rangle$.

\end{thm}

\begin{proof} Given a Harsanyi type space $\langle \Omega, \mathcal{A}, T\rangle$, we can define a partition on $\Omega$: for each $w\in \Omega$,
    \begin{center}
        $\Pi(w): =[T(w)]$
    \end{center}
In other words, the equivalence class containing $w$ is the set of all states whose probability measure is the same as that in $w$.  It is easy to check that $\langle \Omega, \mathcal{A}, T, \Pi\rangle$ is a knowledge-belief system.
\end{proof}

\begin{cor}  The logic $\Sigma_{HK}$ is a conservative extension of $\Sigma_H$.
\end{cor}

\begin{proof} This follows immediately from Theorem \ref{Completeness}, Theorem \ref{CompletenessKnowledgeBelief} and the above conservation theorem \ref{ConservationKnowledgeBelief}.
\end{proof}

\section{Related works and Conclusions}

It is interesting to note the similarity of our axiomatization $\Sigma_+$ in \cite{Zhou09} to the list of properties about semantic belief operators $B^{\geq r}$ in the Introduction.     Despite the similarity, our proof of the completeness of $\Sigma_+$ \cite{Zhou07,Zhou09} is in keeping with the Kripke-style proof of completeness in modal logic \cite{BRV00}, and  is quite different from Samet's analytical argument in \cite{Sam00} about semantical operators $B^{\geq r}$. Our axiomatization is in the spirit of modal and coalgebraic logic \cite{Moss99}.
As in ordinary modal logic, our language is \emph{finitary} in the sense that we allow only for \emph{finite} conjunctions and disjunctions.
So it may look strange to show weak completeness by including such an infinitary rule as $(ARCH)$. But this strangeness can be explained away in our completeness proof which is given through a \emph{probabilistic version} of standard filtration method in modal logic \cite{Zhou07}.  In this method, the crucial step is to define a probability measure at each state in the finite canonical model for a given consistent formula $\phi$.  In order to achieve this, we  extend every maximally consistent set of formulas by decreasing the granularity of rational indices in the language for filtration.  This extension is possible \emph{only} by including $(ARCH)$.  This is due to the Archimedean property for probability indices in the language.
Moreover, one may apply Proposition 3.2 in \cite{HM01} to show the necessity of $(ARCH)$ for the weak completeness of $\Sigma_+$.   Since all of our contributions in the paper are related to weak completeness and have nothing to do with strong completeness, we decide not to discuss strong completeness here. One may refer to \cite{KozenLMP13, KozenMP13, ZhouY12, MardareCL12,Zhou11} for details about this topic. Also it may be interesting to note the close relation of our Archimedean rule to the approximation of Markov processes \cite{DGJP03,ZhouY12,ChaputDPP09}.

Our above system is motivated by the work by Fagin et.al. in \cite{FHM90, FH94} and more directly the work by Heifetz and Mongin in \cite{HM01}. However, our approach to completeness \cite{Zhou07,Zhou09} differs from theirs in that we don't use any arithmetic formulas like
reasoning about linear inequalities in \cite{FH94} or any arithmetic
style rule like the following rule (B) in \cite{HM01}. Let $(\phi_1,
\cdots, \phi_m)$ be a finite sequence of formulas and $\phi^{(k)}$
denote $\bigvee_{1\leq i_1 <\cdots <i_k\leq n}(\phi_{i_1}\wedge
\cdots \wedge \phi_{i_k})$.
 $\bigwedge_{k=1}^m (\phi^{(k)}\leftrightarrow \psi^{(k)})$
is denoted as $(\phi_1, \cdots, \phi_m)\leftrightarrow (\psi_1,
\cdots, \phi_m)$.  The inference rule (B) can be stated as follows:

\begin{center}
$(B): \rool{ ((\phi_1, \cdots, \phi_m)\leftrightarrow (\psi_1,
\cdots, \psi_m))}{ ((\bigwedge_{i=1}^m L_{r_i}\phi_i)\wedge
(\bigwedge_{j=2}^m M_{s_j}\psi_j))\rightarrow L_{(r_1+\cdots +
r_m)-(s_2+\cdots + s_m)}\psi_1}{\mbox ,}$
\end{center}

 \noindent for  $(r_1+\cdots + r_m)-(s_2+\cdots + s_m)\in [0, 1]$.

In this paper we have shown some important properties for probability logic for Harsanyi type spaces. First, we present an axiomatization $\Sigma_H$ for the class of Harsanyi type spaces, which is different from those in the literature in our infinitary Archimedean rule (ARCH), and employ a probabilistic version of filtration method to show its completeness. By using the same method, we show both a denesting property and a unique extension theorem for $\Sigma_H$.  Moreover, we apply these properties to show that the canonical model of $\Sigma_H$ for a single agent in a finite local language is finite.  In contrast, we prove by demonstrating a continuum of different $J^r$-lists through a special Harsanyi type space called bi-sequence space that, for at least 2 agents, the canonical space has the continuum of different states.  The difference between the cardinalities of these canonical spaces illustrates the relative complexity of multi-agent interactive epistemology compared with the one-single-agent epistemology from the perspective of probabilistic beliefs. Aumann's knowledge-belief systems are Harsanyi type spaces with an appropriate knowledge operator. Finally we formulate the relationship between knowledge and probabilistic beliefs in Aumann's knowledge-belief systems by generalizing the three notions of definability in multi-modal logic \cite{HSS09a}.

Our logical characterization of the relationship between knowledge and probabilistic beliefs is motivated by Aumann's knowledge-belief systems in \cite{Aum99b}.  Aumann has constructed a canonical knowledge-belief space, that is, a knowledge-belief space which contains all maximally $\Sigma_{HK}$-consistent sets of formulas in the language $\mathcal{L}_{HK}$.   His construction of the canonical model inspired Heifetz and Samet's work showing the existence of a universal (Harsanyi) type space \emph{without} any topological assumptions \cite{HS98}.  Following a similar line, Meier proved the existence of a universal knowledge-belief structures under the condition that the knowledge operators of the agents in a knowledge-belief space operate \emph{only on measurable subsets} of the space \cite{Meier08}.  Not every knowledge-operator is induced by a partition \cite{Samet10}.  So knowledge as an operator is related to but different from knowledge induced by a partition. This difference can be explained by the two different semantics in modal logic: algebraic semantics vs frame-based semantics \cite{BRV00}.  Actually, the above three notions of definability in multi-modal logics are characterized in \emph{algebraic} semantics not in frame-based semantics as in Aumann's original formulation of knowledge-belief systems. In the present context, a pressing problem is to understand  in knowledge-belief systems the relationship of knowledge as operator and knowledge induced by a partition in order to characterize the interaction between knowledge-hierarchy \cite{HS98a} and belief hierarchy \cite{HS98} in these systems.  We also expect to apply our logic for Harsanyi type spaces to analyze finite depth of reasoning and equilibrium  in
games with incomplete information \cite{Kets13}.

\section*{Acknowledgements}

This paper is based on my PhD dissertation \cite{Zhou07} from Indiana University.  I am more than grateful to my PhD
advisors Professors Jon Michael Dunn and Lawrence Moss for their
constant guidance. Also I would like to thank  Professors  Ernst-Erich Doberkat, Daniel
Leivant, Ingo Battenfeld, Pedro S{\'a}nchez Terraf, Jason Teutsch, Cornelia van Cott,  and Joshua Sack for helpful comments at different stages of this research.  Finally I want to thank two referees for their critical comments to help improve an earlier version of this paper.

\bibliographystyle{plain}
\bibliography{EED}

\section*{Appendix: Some Proofs}

\begin{proof} (Proof of Lemma \ref{MainLemma}) Our proof here is adapted from the Crucial Lemma 4.5 in \cite{HS98}.
    It is easy to see that $(2)\Rightarrow (1)$. Now we show the other direction. Let
    $\mathcal{F}_0$ be the $\sigma$-algebra generated by the set $\{B^{\geq r}(E): E\in \mathcal{A}_0, r\in [0, 1]\}$ and $\mathcal{F}$ be the $\sigma$-algebra generated by the set $\{B^{\geq r}(E): E\in \mathcal{A}, r\in [0, 1]\}$.  In order to show this lemma, it suffices to prove that $\mathcal{F}_0=\mathcal{F}$. Indeed, for $E\in \mathcal{A}$,  if $\mathcal{F}=\mathcal{F}_0$, $B^{\geq r}(E)\in \mathcal{F}_0$. On the other hand, (1) implies that $\mathcal{F}_0\subseteq \mathcal{A}$. So it follows that $B^{\geq r}(E)\in \mathcal{A}$ and hence (2) is proved.
    Define $\mathcal{A}'=\{E\in \mathcal{A}: B^{\geq r}(E)\in \mathcal{F}_0$ for all r$\}$. It is easy to see that $\mathcal{A}_0\subseteq \mathcal{A}'$.  According to Halmos' monotone class theorem, in order to show that $\mathcal{F}\subseteq \mathcal{F}_0$, it suffices to show that $\mathcal{A}'$ is a monotone class.  Indeed, if we show that $\mathcal{A}'$ is a monotone class, then the $\sigma$-algebra $\mathcal{A}$ generated by $\mathcal{A}_0$ is also a subset of $\mathcal{A}'$, which implies that $\mathcal{F}\subseteq \mathcal{F}_0$.

    \begin{clm} If $(E_n)_{n=1}^{\infty}$ be a decreasing sequence of events in $\mathcal{A}'$, then $\bigcap_{n=1}^{\infty}E_n\in \mathcal{A}'$.
    \end{clm}
    Proof of the claim: For each $w
    \in \Omega$, since $T(w, \cdot)$ is a probability measure , $T(w, E_n)$ converges to $T(w,\bigcap_{n=1}^{\infty}E_n )$. This implies the following equality:
        \begin{center}
            $\{w\in \Omega: T(w, \bigcap_{n=1}^{\infty}E_n)\geq r\}= \bigcap_{n=1}^{\infty}\{w\in \Omega: T(w, E_n)\geq r\}$
        \end{center}

    Simply, $B^{\geq r}(\bigcap_{n=1}^{\infty}E_n)=\bigcap_{n=1}^{\infty}B^{\geq r}(E_n)\in \mathcal{F}_0$. It follows that $\bigcap_{n=1}^{\infty}E_n\in \mathcal{A}'$.

    \begin{clm} If $(E_n)_{n=1}^{\infty}$ be a increasing sequence of events in $\mathcal{A}'$, then $\bigcup_{n=1}^{\infty}E_n\in \mathcal{A}'$.
    \end{clm}
    Proof of the claim: For each $w
    \in \Omega$, since $T(w, \cdot)$ is a probability measure , $T(w, E_n)$ converges to $T(w,\bigcup_{n=1}^{\infty}E_n )$. This implies the following equality:
        \begin{center}
            $\{w\in \Omega: T(w, \bigcup_{n=1}^{\infty}E_n)\geq r\}= \bigcap_{m=1}^{\infty}\bigcup_{n=1}^{\infty}\{w\in \Omega: T(w, E_n)\geq r-1/m\}$
        \end{center}

    Simply, $B^{\geq r}(\bigcup_{n=1}^{\infty}E_n)=\bigcap_{m=1}^{\infty}\bigcup_{n=1}^{\infty}B^{\geq r-1/m}(E_n)\in \mathcal{F}_0$. It follows that $\bigcup_{n=1}^{\infty}E_n\in \mathcal{A}'$.

    So we have shown the lemma.
\end{proof}

\begin{proof}(Proof of Lemma \ref{Extension}) First we show that $\mathcal{A}$ is a $\Sigma_H$ algebra.
 \begin{enumerate}
 \item (A1) For any $E\in \mathcal{B}, B^{\geq 0}E = W$.
 \item  (A2) $B^{\geq r}W = W$ for all $r\in [0, 1]$.
 \item  (A3) We divide this case into several subcases:
    \begin{itemize}
        \item  Case 1: $E_1\in \mathcal{U}$ and $E_2\in \mathcal{U}$.  Assume that $r,s>0$. In this case, $E_1\cap E_2\in \mathcal{U}, E_1\cap \neg E_2 \not \in \mathcal{U}$, hence $B^{\geq r}(E_1\cap E_2) = (E_1\cap E_2)\cup \{0\}$, $B^{\geq s}(E_1\cap \neg E_2) = (E_1\cap \neg E_2)\setminus\{0\}$ and $B^{\geq r+s}(E_1) = E_1\cup \{0\}$. Obviously $B^{\geq r}(E_1\cap E_2)\cap B^{\geq s}(E_1\cap \neg E_2) = \emptyset\subseteq E_1\cup \{0\} =B^{\geq r+s}(E_1)$.  When $r=0$( $s=0$), $B^{\geq r}(E_1\cap E_2) = W (B^{\geq s}(E_1\cap \neg E_2) = W)$. Definitely, we have $B^{\geq s}(E_1\cap \neg E_2)\subseteq B^{\geq r+s}(E_1)(B^{\geq r}(E_1\cap E_2)\subseteq B^{\geq r+s}(E_1))$.
        \item Case 2: $E_1\in \mathcal{U}$ and $E_2\not\in \mathcal{U}$. The proof of this case is similar to that of the above case.
        \item Case 3: $E_1\not\in \mathcal{U}$ and $E_2\in \mathcal{U}$. Assume that $r > 0$ and $s>0$. In this case, $E_1\cap E_2\not\in \mathcal{U}, E_1\cap \neg E_2 \not \in \mathcal{U}$, hence $B^{\geq r}(E_1\cap E_2) = (E_1\cap E_2)\setminus \{0\}$, $B^{\geq s}(E_1\cap \neg E_2) = (E_1\cap \neg E_2)\setminus\{0\}$ and $B^{\geq r+s}(E_1) = E_1\setminus \{0\}$. Obviously $B^{\geq r}(E_1\cap E_2)\cap B^{\geq s}(E_1\cap \neg E_2) = \emptyset\subseteq E_1\setminus \{0\} =B^{\geq r+s}(E_1)$. When $r=0$( $s=0$), $B^{\geq r}(E_1\cap E_2) = W (B^{\geq s}(E_1\cap \neg E_2) = W)$. Definitely, we have $B^{\geq s}(E_1\cap \neg E_2)\subseteq B^{\geq r+s}(E_1)(B^{\geq r}(E_1\cap E_2)\subseteq B^{\geq r+s}(E_1))$.
        \item Case 4: $E_1\not\in \mathcal{U}$ and $E_2\not\in \mathcal{U}$. The proof of this case is similar to that in Case 3.

    \end{itemize}

    \item (A4) Similarly we divide the proof into the following several cases:
        \begin{itemize}
            \item Case 1: $E_1\in \mathcal{U}$ and $E_2\in \mathcal{U}$.  Assume that $r,s>0$. In this case, $E_1\cap E_2\in \mathcal{U}, E_1\cap \neg E_2 \not \in \mathcal{U}$, hence $\neg B^{\geq r}(E_1\cap E_2) = \neg [(E_1\cap E_2)\cup \{0\}]= [(\neg E_1)\cup(\neg E_2)]\setminus\{0\}$, $\neg B^{\geq s}(E_1\cap \neg E_2) = \neg [(E_1\cap \neg E_2)\setminus\{0\}]= (\neg E_1)\cup E_2\cup \{0\}$ and $\neg B^{\geq r+s}(E_1) = \neg [E_1\cup \{0\}] = (\neg E_1)\setminus\{0\}$. Obviously $\neg B^{\geq r}(E_1\cap E_2)\cap \neg B^{\geq s}(E_1\cap \neg E_2) = (\neg E_1) \setminus\{0\}= B^{\geq r+s}(E_1)$.  When $r=0$( $s=0$), $B^{\geq r}(E_1\cap E_2) = W (B^{\geq s}(E_1\cap \neg E_2) = W)$. Definitely, we have $\neg B^{\geq r}(E_1\cap \neg E_2)=\emptyset \subseteq B^{\geq r+s}(E_1)(\neg B^{\geq s}(E_1\cap E_2)\subseteq\neg B^{\geq r+s}(E_1))$.
            \item Case 2: $E_1\in \mathcal{U}$ and $E_2\not \in \mathcal{U}$. The proof of this case is similar to that of the above case.
            \item Case 3: $E_1\not \in \mathcal{U}$ and $E_2\in \mathcal{U}$.  Assume that $r,s>0$. In this case, $E_1\cap E_2\not \in \mathcal{U}, E_1\cap \neg E_2 \not \in \mathcal{U}$, hence $\neg B^{\geq r}(E_1\cap E_2) = \neg [(E_1\cap E_2)\setminus \{0\}]= [(\neg E_1)\cup(\neg E_2)]\cup\{0\}$, $\neg B^{\geq s}(E_1\cap \neg E_2) = \neg [(E_1\cap \neg E_2)\setminus\{0\}]= (\neg E_1)\cup E_2\cup \{0\}$ and $\neg B^{\geq r+s}(E_1) = \neg [E_1\setminus \{0\}] = (\neg E_1)\cup\{0\}$. Obviously $\neg B^{\geq r}(E_1\cap E_2)\cap \neg B^{\geq s}(E_1\cap \neg E_2) = (\neg E_1) \cup\{0\}= B^{\geq r+s}(E_1)$. When $r=0$( $s=0$), $B^{\geq r}(E_1\cap E_2) = W (B^{\geq s}(E_1\cap \neg E_2) = W)$.Definitely, we have $\neg B^{\geq r}(E_1\cap \neg E_2)=\emptyset \subseteq B^{\geq r+s}(E_1)(\neg B^{\geq s}(E_1\cap E_2)\subseteq\neg B^{\geq r+s}(E_1))$.
             \item Case 4: $E_1\not\in \mathcal{U}$ and $E_2\not\in \mathcal{U}$. The proof of this case is similar to that in Case 3.
        \end{itemize}

        \item (A5)  We divide the proof into the following subcases:
            \begin{itemize}
                \item Case 1: $E\in \mathcal{U}$. Assume that $r+s>1$. Since $0\leq r, s \leq 1$, $r>0$ and $s> 0$. It follows that $B^{\geq r} E = E\cup \{0\}$ and $\neg B^{\geq s}(\neg E) = \neg ((\neg E) \setminus\{0\})$.  Obviously, $B^{\geq r} E= E\cup \{0\}\subseteq \neg (B^{\geq s}(\neg E))$.
                \item Case 2: $E\not\in \mathcal{U}$. Assume that $r+s>1$. Since $0\leq r, s \leq 1$, $r>0$ and $s> 0$. It follows that $B^{\geq r} E = E\setminus \{0\}$ and $\neg B^{\geq s}(\neg E) = \neg ((\neg E) \cup\{0\})$.  Obviously, $B^{\geq r} E= E\setminus \{0\}\subseteq \neg (B^{\geq s}(\neg E))$.
            \end{itemize}
        \item (DIS) Assume that $E_1 \subseteq E_2$.  It follows directly that $B^{\geq r} E_1 \subseteq B^{\geq r} E_2$.
        \item (ARCH) Assume that $E\subseteq B^{\geq s} (E')$ for all $s< r$. Obviously $r >0$ and hence $r/2 > 0$.  Since $0 < r/2 < r$, $B^{\geq r} (E') = B^{\geq r/2}( E')$. It follows that $E \subseteq B^{\geq r}(E')$.
        \item (H1)  We divide the proof of this case into the following 3 subcases:
            \begin{itemize}
                \item Case 1: $r =0$. In this case,  it is obvious that $B^{\geq r} (E) = W = B^{\geq 1}(B^{\geq r}(E))$.
                \item Case 2: $r>0$ and $E\in \mathcal{U}$. It follows that $B^{\geq r} (E) = E\cup \{0\} = B^{\geq 1}(B^{\geq r}(E))$.
                \item Case 3: $r>0$ and $E\not\in \mathcal{U}$. It follows that $B^{\geq r} (E) = E\setminus \{0\} = B^{\geq 1}(B^{\geq r}(E))$.
                \end{itemize}
        \item (H2)  We divide the proof of this case into the following 3 subcases:
            \begin{itemize}
                \item Case 1: $r =0$. In this case,  it is obvious that $\neg B^{\geq r} (E) = \emptyset = B^{\geq 1}(\neg B^{\geq r}(E))$.
                \item Case 2: $r>0$ and $E\in \mathcal{U}$. It follows that $\neg B^{\geq r} (E) = \neg E\setminus \{0\} = B^{\geq 1}(\neg B^{\geq r}(E))$.
                \item Case 3: $r>0$ and $E\not\in \mathcal{U}$. It follows that $\neg B^{\geq r} (E) = \neg E\cup \{0\} = B^{\geq 1}(\neg B^{\geq r}(E))$.

            \end{itemize}

 \end{enumerate}
  Finally we have shown all the cases and hence have show that $\mathcal{A}$ is indeed a $\Sigma_H$-algebra.  Now we show that it can not be extended to a  $\Sigma_{HK}$-algebra. We prove this by contradiction.

  Suppose that it were. It is easy to check that the following propositions hold.
\begin{itemize}
    \item For any event $E\in \mathcal{U}$ and $0\in E$, $K(E)= B^{=1}E = E$;
    \item For any event $E\not\in \mathcal{U}$ and $0\not\in E$, $K(E)= B^{=1}E = E$;
    \item For any event $E\in \mathcal{U}$ and $0\not\in E$, $K(E)=  E \subsetneq B^{=1}(E) = E\cup \{0\}$;
    \item For any event $E\not\in \mathcal{U}$ and $0\in E$, $K(E)= B^{=1}E = E\setminus\{0\}$.
\end{itemize}
Note that, for any event $E$, $B^{=1}(K(E))= K(E)$.  For any event $E_3$ such that $E_3\in \mathcal{U}$ and $0\not\in E_3$,
\begin{center}
    $B^{=1}E_3= B^{=1}(K(E_3))= K(E_3)$
\end{center}
But this contradicts the fact that $B^{=1}E_3 \neq  K(E_3)$.  So we have shown the second part of this lemma.
 \end{proof}

\end{document}